\documentclass[11pt,a4paper]{article}
\usepackage[utf8]{inputenc}
\usepackage[T1]{fontenc}
\usepackage{amsmath}
\usepackage{amsfonts}
\usepackage{amssymb}
\usepackage{amsthm}
\usepackage{bm}
\usepackage{enumitem}
\setlist{itemsep=1pt,topsep=2pt,parsep=1pt}
\usepackage{upgreek}
\usepackage{tikz}
\usepackage{authblk}

\usepackage{xcolor}
\usepackage{pgfplots}
\usepgfplotslibrary{fillbetween}

\theoremstyle{plain}
\newtheorem{theorem}{Theorem}[section]

\newtheorem{proposition}[theorem]{Proposition}

\newtheorem{lemma}[theorem]{Lemma}

\newtheorem{remark}[theorem]{Remark}
\newtheorem{guess}[theorem]{Guess}

\newtheorem*{remarque*}{Remark}

\newcommand{\bbP}{\mathbb{P}}
\newcommand{\bbE}{\mathbb{E}}
\newcommand{\bbR}{\mathbb{R}}
\newcommand{\bbN}{\mathbb{N}}
\newcommand{\cZ}{\mathcal{Z}}
\newcommand{\cP}{\mathcal{P}}
\newcommand{\cF}{\mathcal{F}}

\renewcommand{\epsilon}{\varepsilon}
\newcommand{\gep}{\epsilon}
\newcommand{\gp}{\varphi}

\newcommand{\dr}{\mathrm{d}}
\newcommand{\dd}{\mathrm{d}}

\newcommand{\ind}{\bm{1}}

\newcommand{\sumtwo}[2]{\sum_{\substack{#1 \\ #2}}} 

\renewcommand{\tilde}{\widetilde}

\makeatletter
\def\namedlabel#1#2{\begingroup
   \def\@currentlabel{#2}%
   \label{#1}\endgroup
}

\usepackage[left=2.5cm,right=2.5cm,top=2cm,bottom=2cm]{geometry}
\usepackage[colorlinks=true,
            linkcolor=black,
            urlcolor=black,
            citecolor=black]{hyperref}
\numberwithin{equation}{section}

\newcommand{\black}{\color{black}}

\usepackage{tocloft}
\author[*,$\dagger$]{Quentin Berger}
\author[$\ddagger$]{Lo\"ic B\'ethencourt}
\author[*]{Camille Tardif}
\affil[*]{\footnotesize Laboratoire de Probabilit\'e Statistique et Mod\'elisation, Sorbonne Universit\'e.}
\affil[$\dagger$]{\footnotesize D\'epartement de Math\'ematiques et Applications, \'Ecole Normale Sup\'erieure, PSL.}
\affil[$\ddagger$]{\footnotesize
Université Côte d’Azur, CNRS, LJAD, France.}

\title{On joint returns to zero of Bessel processes}

\date{}

\begin{document}

\maketitle

\vspace{-25pt}

\begin{abstract}
\noindent
In this article, we consider joint returns to zero of $n$ Bessel processes ($n\geq 2$): our main goal is to estimate the probability that they avoid having joint returns to zero for a long time. 
More precisely, considering $n$ independent Bessel processes $(X_t^{(i)})_{1\leq i \leq n}$ of dimension $\delta \in (0,1)$, we are interested in the first joint return to zero of any two of them:
\[
H_n := \inf\big\{ t>0,  \exists  1\leq i <j \leq n \text{ such that } X_t^{(i)} = X_t^{(j)} =0 \big\}  \,.
\]
We prove the existence of a persistence exponent $\theta_n$ such that $\bbP(H_n>t) = t^{-\theta_n+o(1)}$ as $t\to\infty$, and we provide some non-trivial bounds on $\theta_n$.
In particular, when $n=3$, we show that $2(1-\delta)\leq \theta_3 \leq 2 (1-\delta) + f(\delta)$ for some (explicit) function $f(\delta)$ with $\sup_{[0,1]} f(\delta) \approx 0.079$.
\end{abstract}

\begin{figure}[htbp]
\begin{center}
\includegraphics[scale=0.5]{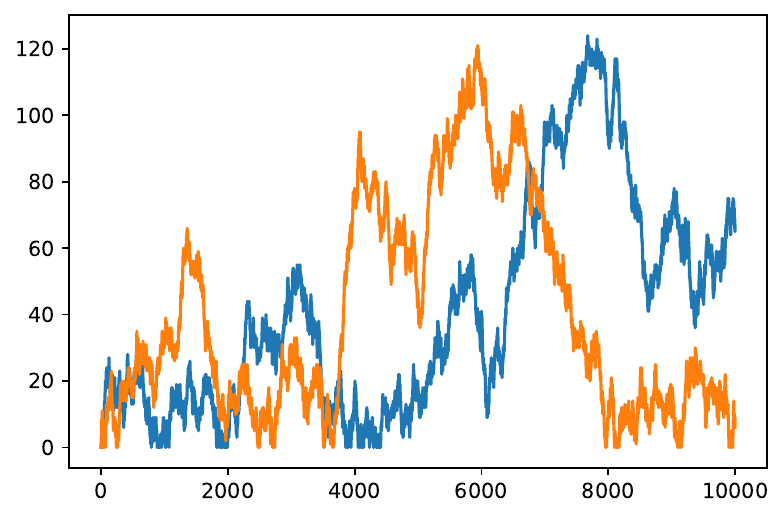}\qquad
\includegraphics[scale=0.5]{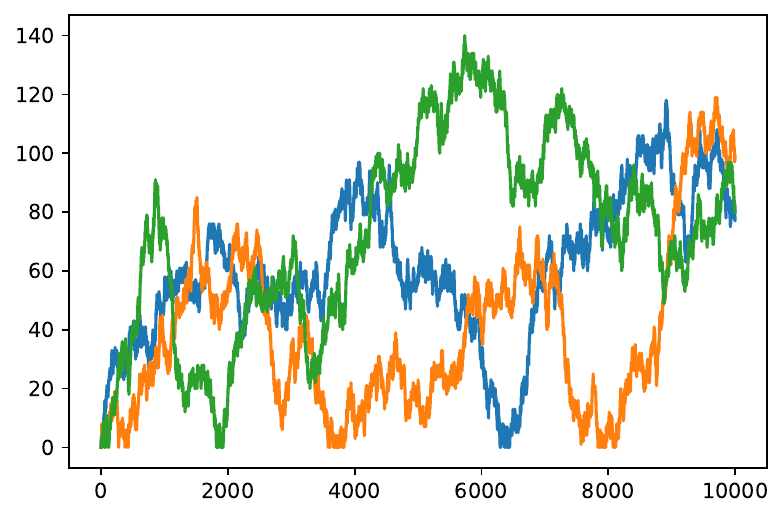}
\end{center}
\vspace{-\baselineskip}
\caption{\small On the left, two Bessel processes of dimension $\delta=3/4$ conditioned not to have joint returns to zero. On the right, three Bessel processes of dimension $\delta=3/4$ conditioned not to have (any) joint returns to zero.}
\end{figure}

\vspace{-10pt}

{\small
\setcounter{tocdepth}{2}

\tableofcontents
}

\section{Introduction and main result}

Let $n\geq 2$ be some fixed integer, and consider $X=(X^{(i)})_{1\leq i \leq n}$ independent squared Bessel processes of dimension $\delta$, 
\textit{i.e.}\  described by the evolution equations
\begin{equation}
\label{def:Bessels}
X_0^{(i)} =x_i\,,\quad 
\dd X_t^{(i)} = 2 \sqrt{X_t^{(i)}} \,\dd W_t^{(i)} + \delta \dd t \,,
\end{equation}
with $(W_t^{(i)})_{1\leq i \leq n}$ independent standard Brownian motions.
In other words, $(X_t)_{t\geq 0}$ is a diffusion in $(\bbR_+)^n$ with generator
\begin{equation}
\label{def:generator}
\mathcal{L}^{n}_{\delta} := 2 \sum_{i=1}^n x_i \frac{\partial^2}{\partial x_i^2} + \delta \frac{\partial}{\partial x_i} \,.
\end{equation}
We denote by $\bbP_x$ the law of $(X_t)_{t\geq 0}$ started from $x = (x_1,\ldots,x_n)$.
For $i,j \in \{1,\ldots, n\}$, $i\neq j$, let us denote
\[
\begin{split}
T_i & := \inf\big\{ t\geq 0, X_t^{(i)} =0\big\}  \,,\\ 
T_{i,j} & := \inf\big\{ t\geq 0, X_t^{(i)} = X_t^{(j)} =0 \big\}  = \inf\big\{ t\geq 0, X_t^{(i)} + X_t^{(j)} =0 \big\} \,,
\end{split}
\]
which are respectively the first return to $0$ of $X^{(i)}$
and the first \textit{joint return to $0$} of $X^{(i)}$ and $X^{(j)}$.
Then, it is classical, see e.g.\ \cite[p.~511]{MR4302463}, to obtain that if $\delta <2$, then for any fixed $i$,
\begin{equation}
\label{eq:oneprocess}
\bbP_{x}(T_i >t) \sim c_{x_i}\, t^{- (1- \frac12\delta)} \qquad \text{ as } t\to\infty\,,
\end{equation}
where the constant $c_{x_i}$ depends only on $x_i$ and $\delta$, and is given by
$
 c_{x_i} = \frac{x_i^{1 -\delta/2}}{2^{1-\delta/2}(1-\delta /2)\Gamma(1-\delta/2)} \,,
$
where $\Gamma$ is the usual gamma function. 

As far as \textit{joint returns} are concerned, for any fixed $i\neq j$, we have, if $\delta\in (0,1)$,
\begin{equation}
\label{eq:twoprocesses}
\bbP_x(T_{i,j}>t) \sim c_{x_i+x_j}\, t^{- (1-\delta)} \qquad \text{ as } t\to\infty\,.
\end{equation}
This can be viewed from the fact that $X^{(i)}+X^{(j)}$ is a squared Bessel process of dimension $\tilde \delta=2\delta$ with starting point $x_i+x_j$, see~\cite[Chap.~XI, Thm~1.2]{RY99}, so one can apply~\eqref{eq:oneprocess}.
We also refer to Section~\ref{sec:twoBessels} for more comments.

\smallskip
In this article, we consider the first joint return to $0$ of \textit{any two} of the $n$ squared Bessel processes, namely
\begin{equation}
\label{def:T}
H_n := \min\big\{ T_{i,j} ,\ 1\leq i < j \leq n \big\} \,.
\end{equation}
This can also be seen as the hitting time of the $(n-2)$-dimensional set $\mathcal{A} = \bigcup_{i\neq j} \{x_i=x_j=0\}$ by the $(\mathbb{R}_+)^n$-valued process $(X_t)_{t\geq 0}$.
In dimension $n=2$, this corresponds to the hitting time of the corner of the quadrant $(\bbR_+)^2$; in dimension $n=3$, this is the hitting time of one of the axis of the octant $(\bbR_+)^3$.

\begin{remark}
\label{rem:k}
One could also consider the hitting time by $(X_t)_{t\geq 0}$ of the $(n-k)$-dimensional set $\mathcal{A}^{(k)} = \bigcup_{|I|=k} \{x_i=0 \, \forall i\in I\}$, corresponding to simultaneous returns to $0$ of $k$ Bessel processes. We will make a few comments on this general case, but for simplicity we focus on the case $k=2$ in the rest of the paper.
\end{remark}

Our main goal is to estimate the tail probability $\bbP_x(H_n>t)$ as $t\to\infty$.
We will focus on the case where $\delta$ is in $(0,1)$, since in the case $\delta\geq 1$ we have $T_{i,j} =+\infty$ a.s.\ for all $i,j$, while for $\delta\leq 0$ squared Bessel processes are absorbed at $0$ (still, we discuss this case in Section~\ref{sec:delta-neg}).

We prove below that the \emph{persistence exponent} $\theta_n$ exists, see Proposition~\ref{prop-existence-theta}, \textit{i.e.}\ that we have, for any $x\notin \mathcal{A}$,
\[
\bbP_x(H_n>t) = t^{- \theta_n +o(1)} \quad \text{ as } t\to\infty \,.
\] 
The question is then to identify $\theta_n$;
a further question would be to obtain a sharper asymptotic behavior, for instance $\bbP_x(H_n>t) \sim c_{n,x}\, t^{-\theta_n}$.

In this article, we put some emphasis on the case $n=3$ for simplicity. Even if we are not able to determine the exponent~$\theta_3$, we prove non-trivial upper and lower bounds, showing that $2(1-\delta)\leq \theta_3 \leq 2 (1-\delta) + f(\delta)$ for some (explicit) function $f(\delta)$ with $\sup_{[0,1]} f(\delta) \approx 0.079$, see Theorem~\ref{thm:main} below.

\subsection{Some motivations}
\label{sec:motiv}

\paragraph{Spatial population with seed-bank and renewal processes.}
Our original motivation was a question raised by F.~den Hollander, in the context of renewal processes, in relation to models of populations with seed-banks \cite{BEGK15,BGKW16}, in particular in a multi-colony setting, see e.g.~\cite{GdHO22} (or the introduction of \cite{Oomen21} for an overview).
In these models, individuals can become dormant and stop reproducing and after some (random, possibly heavy-tailed) time they wake up, become active and start reproducing but only for a short period of time.
Roughly speaking, the times where individual from a seed-bank becomes active form a renewal process, and \textit{joint renewals} correspond to times when individuals become \textit{jointly active} and are able to interact and exchange genetic material.

Thus, understanding the tail behavior of the joint renewals is key in understanding the evolution of genetic variability in these models.
Our question would then amount to studying the tail probability of having no joint renewals for  \textit{any two individuals} in a given set of $n$ individuals.

\paragraph{Renewal processes on $\mathbb{N}$ and \textit{joint renewals}.}
Let us formulate the question of the previous paragraph directly in terms of renewal processes and make some comments.
Consider $n$ independent \emph{recurrent} renewal processes $(\tau^{(i)})_{1\leq i\leq n}$ on $\mathbb{N}_0=\{0,1,2,\ldots\}$: $\tau^{(i)} = \{\tau_k^{(i)}\}_{k\geq 0}$ is such that $\tau_0^{(i)} =0$ and $(\tau_k^{(i)}-\tau_{k-1}^{(i)})_{k\geq 1}$ are i.i.d.\ $\bbN$-valued random variables.
We can interpret $\tau^{(i)}$ as the \textit{activation times} of an individual in a seed-bank, or as the \textit{return times to $0$} of a Markov process.
We assume that $\bbP(\tau_1^{(i)} > t) \sim c_0 t^{-\alpha}$ as $t\to\infty$, for some $\alpha>0$ and some constant $c_0>0$. 
This is a natural \textit{fat tail} assumption for population with seed-bank, see~\cite{BGKS13} and it is also verified for the return times to $0$ of Bessel-like random walks, see~\cite{Ale11}; in particular, the parameter $\alpha$ is related to the dimension $\delta$ of the Bessel-like random walk\footnote{More precisely, $\frac1N \tau^{(i)}$ converges in distribution (as a closed subset of $[0,\infty)$) to a $\min(\alpha,1)$-stable regenerative set, see e.g.~\cite[\S~A.5.4]{Giac07}, which can be interpreted as the \textit{zero set} of a Bessel process of dimension $\delta$.} by the relation $\alpha = 1- \frac12 \delta$, see e.g.~\eqref{eq:oneprocess} (or equivalently $\delta=2(1-\alpha)<2$). 

Defining $\rho^{(i,j)} := \tau^{(i)} \cap \tau^{(j)}$ the \textit{joint renewals} of $\tau^{(i)}$ and $\tau^{(j)}$, then one easily have that $\rho^{(i,j)}$ is also a renewal process, which is recurrent if $\alpha >\frac12$ (which corresponds to $\delta<1$).
In the case $\alpha \in (\frac12,1)$ (which corresponds to $\delta\in (0,1)$), the renewal structure allows one to obtain the tail asymptotic $\bbP(\rho^{(i,j)} \cap (0,t] =\emptyset)$ thanks to a Tauberian theorem, simply by estimating the renewal function $U(t) = \sum_{s=1}^t \bbP(s \in \rho^{(i,j)}) = \sum_{s=1}^t \bbP(s \in \tau^{(i)})^2$: estimates on $\bbP(s \in \tau^{(i)})$ are available (see e.g.~\cite{CD19,GL62}) and after a short calculation one gets that $\bbP(\rho^{(i,j)} \cap (0,t] =\emptyset) \sim c_1 t^{- (2\alpha-1)} = c_1 t^{-(1-\delta)}$; we refer to~\cite{AB16} for details. 
The case $\alpha\geq 1$ is actually more delicate since one cannot apply a Tauberian theorem, but one has $\bbP(\rho^{(i,j)} \cap (0,t] =\emptyset) \sim c_1' t^{- \alpha}$, see~\cite[Thm.~1.3-(iii)]{AB16}. 
We refer to~\cite{AB16} for an overview of results on the intersection of two renewal processes.

However, if there are $n$ renewal processes and if we define the set of \textit{joint renewals} as $\rho := \{ s\in \mathbb{N}_0 ,  \exists 1\leq i<j\leq n ,    s \in \tau^{(i)} \cap \tau^{(j)}\}$, then $\rho$ is \emph{not} a renewal process anymore if $n\geq 3$.
Then, it is not clear how to estimate the tail probability $\bbP(\rho \cap (0,t] = \emptyset)$ and  the goal of the present article is precisely to give an idea on how this probability should decay, since it is natural to expect that $\bbP(\rho \cap (0,t] = \emptyset) \approx \bbP(H_n>t)$, with squared Bessels of dimension $\delta := 1-2\alpha$.

\paragraph{A toy model for collisions of particles.}
Another source of motivation for studying joint returns to $0$ is that one can interpret the instant $T_{i,j}$ as the first \textit{collision time} between two particles $i,j$ --- for instance one could interpret $Y_{i,j}:=X^{(i)}+X^{(j)}$ as the distance between particles $i,j$.
This is of course a toy model of particle systems since particles have not much interaction, but the question is already interesting (and difficult) because of the intricate relation between the processes $Y_{i,j}$.

In the following, we sometimes call an instant $t$ such that $X_t^{(i)}=X_t^{(j)}=0$ a \emph{collision} between particles $i$ and $j$.
In this framework, our question consists in studying the large deviation probability of having \textit{no collision (of any pair of particles)} for a long time.
We have in mind several models where such a question is natural, such as mutually interacting Brownian of Bessel processes\footnote{Also related to Dyson's Brownian motion and Dunkl processes, see e.g.\ \cite{Demni_rev} for an overview.}, see e.g.\ \cite{CL97,CL01,CL07},
or Keller--Segel particles systems, see e.g.\ \cite{FJ17,FT21} --- note that both models feature (squared) Bessel processes.

\subsection{Main results: joint returns to zero of $n\geq 3$ Bessel processes}
\label{sec:mainresult}

We now turn to the case of $n\geq 3$ Bessel processes and state our main result. Recall the definition~\eqref{def:T} of~$H_n$, the hitting time of $\mathcal{A}=\bigcup_{i\neq j} \{x_i=x_j=0\}$.
First of all, we show the existence of the \emph{persistence exponent} $\theta_n$.

\begin{proposition}\label{prop-existence-theta}
There is some $\theta_n \geq 0$, that depends on $\delta$ but not on the starting point $x \in (\bbR_+)^n\setminus \mathcal{A}$, such that 
\[
\lim_{t\to\infty} \frac{1}{\log t} \log \bbP_x( H_n> t) =-\theta_n \,.
\]
In other words, $\bbP_x( H_n> t) = t^{-\theta_n +o(1)}$ as $t\to\infty$.
\end{proposition}

Before we state our main result, let us give ``trivial'' bounds on the probability $\bbP_x(H>t)$, and so on $\theta_n$.
For an upper bound, we can use the independence of $(T_{2i-1,2i+1})_{1\leq i \leq \lfloor n/2\rfloor}$, together with~\eqref{eq:twoprocesses}, to obtain that $\bbP_x(H_n>t) \leq c\, t^{- \lfloor n/2\rfloor (1-\delta)}$ as $t\to\infty$.
Hence, this gives the bound
$\theta_n \geq \lfloor n/2\rfloor (1-\delta)$. 
Let us stress that if $n=3$ this gives that $\theta_3 \geq 1-\delta$, which is simply the exponent obtained when $n=2$; in particular, it is a priori not clear whether one has $\theta_3 > 1-\delta$.

For a lower bound, imposing $T_{1,2}>t$ and $T_i>t$ for $i\geq 3$, using the independence and~\eqref{eq:oneprocess}-\eqref{eq:twoprocesses}, we obtain that $\bbP_x(H_n>t) \geq c\, t^{- (1-\delta) - (n-2) (1-\frac12 \delta)}$ as $t\to\infty$.
This gives the upper bound 
$\theta_n \leq n \big(1-\frac12 \delta \big) -1$. 
In particular, when $n=3$ we get $\theta_3 \leq 2-\frac32 \delta$.

\smallskip
Our main result provides a non-trivial lower bound on $\theta_n$, valid for all $n\geq 3$.
In the case $n=3$, we also find an upper bound on $\theta_3$.
\begin{theorem}
\label{thm:main}
For all $n\geq 3$,  we have that
\[
\theta_n\geq (n-1)(1-\delta)  \,.
\]
When $n=3$ we have the following upper bound 
\[
\theta_3 \leq 2(1-\delta) + f(\delta),
\]
with $f(\delta) = \frac14 \big(\sqrt{(6-5\delta)^2 +8\delta(1-\delta)} -(6-5\delta) \big)$.
\end{theorem}

%

\begin{figure}[ht]
\begin{center}
\includegraphics[scale=0.44]{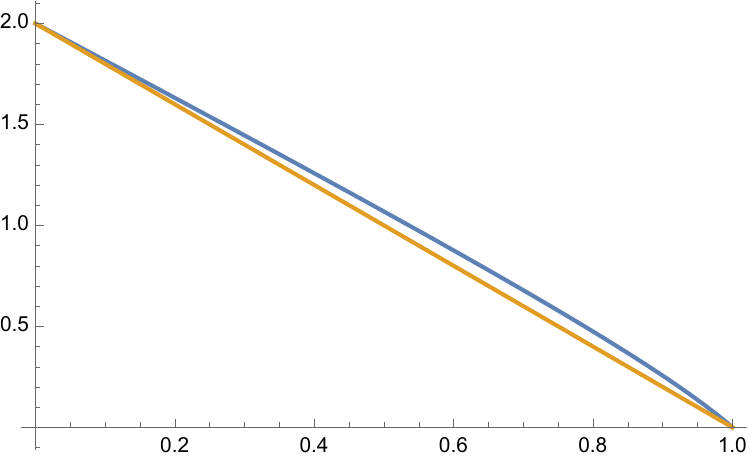}\qquad\qquad
\includegraphics[scale=0.44]{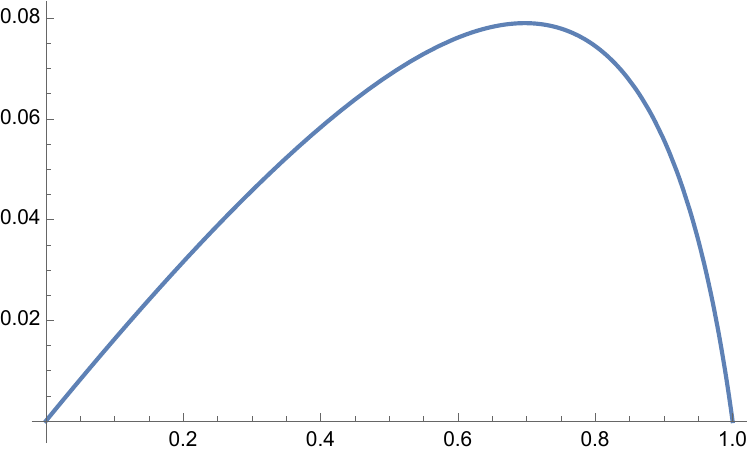}
\end{center}
\caption{\small Illustration of the bounds of Theorem~\ref{thm:main} in the case $n=3$.
The persistence exponent $\theta_3$ verifies $\theta_-:=2(1-\delta) \leq \theta_3 \leq 2(1-\delta)  + f(\delta)=:\theta_+$. The plot on the left-hand side shows the graphs of $\theta_-,\theta_+$ as functions of $\delta\in (0,1)$; the plot on the right-hand side shows the graph of $f$, and numerically one has $\sup_{[0,1]}f(\delta) \approx 0.079$.}
\label{fig:thetas}
\end{figure}

\subsection{First comments and some guesses}

We now make a few comments on our result and we develop some interesting open questions one could pursue.

\paragraph{About $\theta_3$.}
In view of the fact that the function $f(\delta)$ is small (see Figure~\ref{fig:thetas}) and the fact that our upper bound could possibly be improved (see Remark~\ref{rem:upper-better}) one may have the following guess.

\begin{guess}
\label{guess:main}
For $n=3$ and $\delta \in (0,1)$, we have $\theta_3 = 2 (1-\delta)$.
\end{guess}

\noindent
We would not venture to call it a conjecture since we have no simple heuristic as to why this should be the correct answer; in fact we expect that this guess should \emph{not} be correct when $n$ is large, see Guess~\ref{guess:ninfinity} below.

\paragraph{About subsets of joint returns to zero.}
Naturally, there are many other questions one could ask about joint returns to zero of Bessel processes. 
For instance we could consider a subset $K \subset  \cP_2(n) := \{ \{i,j\},  1\leq i <j\leq n \}$ of all possible pairs of indices, and consider $T_{K} := \min\{ T_{i,j} \,, \ \{i,j\}\in K\}$, \textit{i.e.}\ the first joint return for any $X^{(i)}$ and $X^{(j)}$ with $\{i,j\}\in K$.
We focus in this article on the case $K=\cP_2(n)$, and in fact we have no clear guess for a general subset~$K$, even in simple cases such as $n=3$, $K= \{ \{1,2\} ,\{1,3\}\}$.
However, the following guess seems reasonable, but we are not able to prove it.
\begin{guess}
\label{guess:2}
For any $\delta \in(0,1)$, we have $\bbP(T_{1,2}>t, T_{1,3}>t) = t^{-\tilde \theta_3 +o(1)}$ as $t\to\infty$, with $ \theta_2 = 1-\delta < \tilde \theta_3 < 2(1-\delta) \leq \theta_3$.
\end{guess}

This guess somehow tells that it is \textit{strictly harder} to avoid collisions when one considers more pairs of particles, but we are not able to prove any of the bounds $ 1-\delta < \hat \theta_3 < 2(1-\delta)$. 
In fact, our Theorem~\ref{thm:main} shows that it is strictly harder to avoid \emph{any} collision when you have three particles, which is already an achievement.


\paragraph{About $\theta_n$ when $n$ is large.}
Another aspect of the problem one may consider is when the number $n$ of particles is very large. We then have the following guess (for which we give some convincing argument below), which tells in particular that the lower bound $\theta_n \geq (n-1)(1-\delta)$ is \emph{not} sharp, at least when $n$ is large\footnote{Numerical simulations appear to confirm that $\theta_n > (n-1) (1-\delta)$ when $n\geq 5$, at least in some range of $\delta$.}. 

\begin{guess}
\label{guess:ninfinity}
For any fixed $\delta \in (0,1)$, we have that $\theta_n \sim n (1-\frac12\delta)$ as $n\to\infty$.
\end{guess}

Let us briefly explain why we conjecture this specific asymptotic behavior for $\theta_n$.
First of all, we showed a trivial upper bound $\theta_n \leq n (1-\frac12 \delta) -1$ in Section~\ref{sec:mainresult}, which matches this asymptotics.
For the lower bound, the heuristic goes as follows.

First, let us set $\check H_n^{(1)} := \inf\{t, \exists\,  2\leq i \leq n\,,\,  X_t^{(1)}=X_t^{(i)}=0\}$ the first instant of collision of particle number~$1$ with \emph{any} other particle.
Then, we strongly believe (but are not able to prove) that when $n$ is large, one has
\[
\bbP_x(\check H_n^{(1)} > t) = t^{- \check \theta_n +o(1)} \qquad \text{ with }\ \ \check \theta_n \sim 1-\frac12 \delta \,.
\]
Indeed, the easiest way for the particle number~$1$ to avoid a collision with the other $n-1$ particles is to avoid touching~$0$ whatsoever (\textit{i.e.}\ having $T_1>t$), hence the exponent should be close to $1-\frac12 \delta$, which comes from~\eqref{eq:oneprocess}; indeed, requiring all $n-1$ other particles doing something unusual should be much more costly.
With this in mind, we should have that
\[
\bbP_x(H_n>t) = \bbP_x(\check H_n^{(1)} > t) \, \bbP_x( H_n>t \mid \check H_n^{(1)} > t) \approx t^{- \check \theta_n +o(1)} \bbP_x( H_{n-1}>t )\,,
\]
where $H_{n-1}$ denotes the first collision time among $n-1$ particles.
The reasoning here is that the conditioning by the event $\{\check H_n^{(1)} > t\}$ mostly affects the first particle but almost not the others: in practice, we should have $\bbP_x( H_n>t \mid \check H_n^{(1)} > t) \approx \bbP_x( H_n>t \mid T_1 > t) = \bbP_x( H_{n-1}>t)$.
Iterating this argument (as long as the number of particles remains large) supports the guess that $\theta_n = n (1-\frac12\delta) +o(n)$ as $n\to\infty$.

\subsection{Organisation of the rest of the paper}

Let us briefly outline the rest of the paper.
\begin{itemize}[wide]
\item In Section~\ref{sec:comments}, we comment on some related questions:
we present remarkable properties of the case of  $n=2$ Bessel processes (these properties fail for $n\ge 3$);
we give results in the case of a negative dimension $\delta<0$, which are trivial;
we comment on the relation of our question with various PDE problems, which provide a different perspective (that we were not able to exploit).

\item In Section~\ref{sec:prelim}, we present some preliminary results: a comparison theorem that allows us to compare different diffusion processes; a proof of the existence of the persistence exponent $\theta_n$ via an elementary (and general) method (it relies on the sub-additive lemma, with some small additional technical difficulty).
We also present the general strategy of the proof in Section~\ref{sec:strategy}: in a nutshell, the idea is to find an auxiliary process $(Z_t)_{t\geq 0}$ for which $H_n$ is the hitting time of $0$, and to compare $(Z_t)_{t\geq 0}$ with a time-changed Bessel process (for which we know how to control the hitting time of $0$).

\item In Section~\ref{sec:proof}, we implement the strategy outlined in Section~\ref{sec:strategy}.
We introduce two auxiliary processes (a different ones for the lower and the upper bound on $\theta_n$) and compare them with time-changed Bessel processes. 
The time-changes are controlled in a separate Section~\ref{sec:time-changed} (our goal is to give a self-contained and robust proof, and in particular we do not rely on subtle properties of Bessel processes).

\item In Appendix~\ref{app:details_calculs}  and~\ref{app:spectralpb}, we collect some tedious calculations that we had postponed not to break the flow of the proof.
\end{itemize}

\section{Various comments}
\label{sec:comments}

\subsection{About two Bessel processes conditioned on having no joint return to zero}
\label{sec:twoBessels}

Let us now develop a bit on the case of $n=2$ squared Bessel processes, which contains some interesting features and helps understand why the case $n\geq 3$ is more complicated.

A natural approach to attacking the case $n=3$ and a natural question in itself is to consider two Bessel processes \textit{conditioned} on having no collision before time $t$. Indeed one can write 
\begin{equation}
\label{eq:conditioning}
\bbP_x(H_3>t) = \bbP_x(T_{1,2} >t) \bbP_x(H>t \mid T_{1,2}>t) \sim c_x t^{- (1-\delta)} \bbP_x( T_{3,1}, T_{3,2}>t \mid T_{1,2}>t) \,,
\end{equation}
and understanding the behavior of $X_{t}^{(1)}, X_t^{(2)}$ conditioned on $T_{1,2}>t$ seems to be a good start to study $\bbP_x(H_3>t)$.

Interestingly, the behavior of $X_{t}^{(1)}, X_t^{(2)}$ conditioned on having no collision, \textit{i.e.}\ $T_{1,2}=+\infty$, is remarkably clear. 
Indeed, let $S_t:= X_t^{(1)} +X_t^{(2)}$ and $U_t:= X_t^{(1)}/S_t \in [0,1]$, $1-U_t= X_t^{(2)}/S_t$, so that $X_t^{(1)} = S_t U_t$ and $X_t^{(2)} = S_t (1-U_t)$.
Then, a simple application of It\^o's formula gives, after straightforward calculations, that $(S_t)_{t\geq 0}$ and $(U_t)_{t\geq 0}$ satisfy the following SDEs:
\begin{equation}
\label{eq:SU}
\begin{split}
\dd S_t & = 2\sqrt{S_t} \,\dd W_t + 2\delta \dd t \\
\dd U_t & = \frac{2}{\sqrt{S_t}} \sqrt{U_t(1-U_t)} \,\dd \tilde W_t  +  \frac{\delta}{S_t} (1-2U_t)\dd t 
\end{split}
\end{equation}
with $(W_t)_{t\geq 0}$, $(\tilde W_t)_{t\geq 0}$ two \textit{independent} Brownian motions.
In particular, $(S_t)_{t\geq 0}$ is a squared Bessel process of dimension $2\delta$ and $U_t$ can be written as time-changed (by $\int_0^t S_u^{-1} \dd u$) diffusion, independent of $(S_t)_{t\geq0}$.

Hence, conditioning on $T_{1,2} = +\infty$ (\textit{i.e.}\ on $S_t>0$ for all $t>0$) simply has the effect of changing $(S_t)_{t\geq 0}$ to a squared Bessel process of dimension $4-2\delta$, see \cite{GY03}.
We therefore end up with the following result.

\begin{proposition}
\label{prop:conditioned}
Conditionally on $T_{1,2} = +\infty$, the process $(X_t^{(1)},X_t^{(2)})_{t\geq 0}$ have the distribution of $( \tilde S_t \tilde U_{\tilde \tau_t}, \tilde S_t (1-U_{\tilde \tau_t}))_{t\geq 0}$, where  $\tilde S, \tilde U$ are independent diffusion characterized by the following:
\begin{enumerate}
\item $(\tilde S_t)_{t\geq 0}$ is a squared Bessel process of dimension $4-2\delta$, \textit{i.e.}\ follows the evolution equation 
\[
\dd \tilde S_t = 2\sqrt{\tilde S_t}\, \dd W_t  +  (4-2\delta)\dd t \,;
\]

\item $(\tilde U_t)_{t\geq 0}$ follows the evolution equation 
\[
\dd \tilde U_t = 2\sqrt{\tilde U_t(1-\tilde U_t)} \,\dd \tilde W_t  +  \delta (1-2\tilde U_t)\dd t \,;
\]
\end{enumerate}
and $\tilde \tau_t$ is the inverse of $t\mapsto\int_0^t \tilde S_u^{-1} \dd u$.
\end{proposition}

\begin{remark}
We could also define the angle $\Theta_t$, such that $\tilde U_t := \cos^2(\Theta_t)$, $1- \tilde U_t = \sin^2(\Theta_t)$. Applying It\^o's formula, after some calculation one ends up with the following SDE for $(\Theta_t)_{t\geq 0}$:
\[
\dd \Theta_t = \dd \tilde W_t + \frac{\delta-1}{2 \tan(2 \Theta_t)} \, \dd t \,.
\]
Note that it looks like the evolution equation of a Bessel process of dimension $\delta$ when $\Theta$ approaches~$0$ (and similarly for $\frac{\pi}{2} -\Theta$, by symmetry), with a null drift when $\Theta = \frac{\pi}{4}$.
\end{remark}

\noindent
Let us make some further comments and give one result.

\smallskip
\noindent
\textbf{Comment 1.} The conditioning by $\{T_{1,2}=+\infty\}$ significantly changes the behavior of the tail of the first hitting of zero, $\min\{T_1,T_2\}$.
In fact, somewhat surpisingly, the persistence exponent of $\bbP_x( \min\{T_{1},T_2\}>t \mid T_{1,2} =+\infty)$ is equal to~$1$, and in particular it does not depend on $\delta\in(0,1)$.
Indeed,  as $t\to\infty$, we have that,
\[
\bbP_x\big( \min\{T_{1},T_2\}>t \mid T_{1,2} >t \big) = \frac{\bbP_x\big( \min\{T_{1},T_2\}>t  \big)}{\bbP_x\big( T_{1,2} >t \big)} \sim \frac{c_{x_1} c_{x_2} t^{-2(1-\delta/2)}}{ c_{x_1+x_2} t^{-(1-\delta)}} = \frac{c_{x_1} c_{x_2} }{ c_{x_1+x_2}} \; t^{-1} \,,
\]
where we have used~\eqref{eq:oneprocess}-\eqref{eq:twoprocesses};
we leave aside the technicality of replacing the conditioning by $T_{1,2}=+\infty$.
This shows in particular that the conditioning makes it strictly easier for the Bessel processes to avoid hitting zero at all, changing the persitence exponent of $\min\{T_{1},T_2\}$ from~$2-\delta$ to~$1$.


\smallskip
\noindent
{\bf Comment 2.}
The zero set $\cZ^{(i)}:=\{t, X_t^{(i)}=0\}$ of a squared Bessel is a regenerative set, in fact an $\alpha$-stable regenerative set with $\alpha:=1-\frac12 \delta$; see~\cite[Ch.~2]{bertoin1999} for an introduction to regenerative sets.
Then, the set of collision times $\{t, X_t^{(i)} =X_t^{(j)} =0\}$ is $\cZ^{(i)}\cap \cZ^{(j)}$, which is itself a regenerative set.
This regenerative structure is not specific to Bessel processes and holds for any Markov process, and can be useful in estimating the probability $\bbP(T_{i,j}>t) = \bbP(\cZ^{(i)}\cap \cZ^{(j)} \cap[0,t] =\emptyset)$, similarly to the discrete setting (see Section~\ref{sec:motiv} above).
On the other hand, the regenerative structure completely disappears when considering $n\geq 3$ processes, since the set of collision times is then $\bigcup_{i\neq j} \cZ^{(i)}\cap \cZ^{(j)} $ which is not a regenerative set anymore\footnote{Note however that the regenerative structure is present if one considers ``$n$-collisions'', \textit{i.e.}\ simultaneous return to $0$ of the $n$ processes all together.}.

\smallskip
\noindent
{\bf Comment 3.} Proposition~\ref{prop:conditioned} allows us to ``understand'' the law of $\cZ^{(1)}\cup \cZ^{(2)}$ conditioned on $\cZ^{(1)}\cap \cZ^{(2)}=\emptyset$: it is the zero set of the process $R_t := \tilde Y_t \tilde U_{\tilde \tau_t} (1-\tilde U_{\tilde \tau_t})$, for which one has the evolution equation
\[
\dd R_t  =  2\sqrt{R_t (1-3 \widehat{U}_t)} \dd \hat W_t + \delta (1- 6 \widehat{U}_t) \dd t \,,
\]
where $\widehat{U}_t := \tilde U_{\tilde \tau_t}(1-\tilde U_{\tilde \tau_t}) \in [0,\frac14]$.
Note that the process $R_t$ can be interpreted as a time-changed (by $\int_0^t (1-3 \widehat{U}_s) \dd s$) squared Bessel process, with varying dimension $\delta \frac{1- 6 \hat U_t}{1-3 \widehat{U}_t}$ --- the difficulty here is that the variation of the dimension is intricate.

Then, one could hope to understand $\bbP( T_{3,1},T_{3,2} >t \mid T_{1,2} =+\infty)$,  since $\min\{T_{3,1},T_{3,2}\}$ is the hitting time of zero of the process $R_t+X_t^{(3)}$.
In fact, with techniques similar to the ones developed in this paper, we should be able to show that $\bbP_x( T_{3,1},T_{3,2} >t \mid T_{1,2} =+\infty) \leq  t^{-(1-\delta)+o(1)}$ (which in view of~\eqref{eq:conditioning} would correspond to the bound $\theta_3\geq 2(1-\delta)$), but we are not able to obtain matching upper and lower bounds with this approach.

\subsection{The case of a negative dimension}
\label{sec:delta-neg}

Let us comment briefly on the case where the dimension of the squared Bessel processes is negative, \textit{i.e.}\ $\delta\leq 0$.
In that case, the processes $X^{(i)}$ are absorbed at $0$, meaning that $X_t^{(i)} =0$ for all $t\geq T_i$.
Therefore, we get that $T_{i,j} := \max\{T_i,T_j\}$ so that 
\begin{equation}
\label{eq:two-delta-neg}
\bbP_x(T_{i,j} >t) = \bbP_x(T_i >t \text{ or } T_j>t) \sim (c_{x_i}+c_{x_j}) t^{- (1-\frac12 \delta)} \quad \text{ as } t\to\infty \,, 
\end{equation}
using also~\eqref{eq:oneprocess}.
Similarly, we have that $H_n = \min_{1\leq i <j\leq n}\{ T_{i,j}\}$ is the second smallest $T_i$, so we have that
\[
\bbP_x(H_n>t) = \bbP_x\Big( \bigcup_{j=1}^n  \{T_i >t \ \text{ for all } i\neq j  \} \Big) \,.
\]
Using the inclusion-exclusion principle and again~\eqref{eq:oneprocess}, we easily end up with the following result.

\begin{proposition}
\label{prop:delta-neg}
Let $n\geq 2$ and $\delta\leq 0$. Then we have  as $t\to\infty$
\[
\bbP_x(H_n>t) \sim \sum_{j=1}^n \bbP_x\big( T_i >t \ \text{ for all } i\neq j  \big) \sim c_n\, t^{ - \theta_n} \,,
\]
with $\theta_n = (n-1) (1-\delta/2)$ and
with the constant $c_n:= \sum\limits_{j=1}^n \prod\limits_{i\neq j} c_{x_i}$.
\end{proposition}

\noindent
Let us observe that in Proposition~\ref{prop:delta-neg}, the persistence exponent verifies $\theta_n=(n-1) \theta_2$ (similalry as in Guess~\ref{guess:main} with $n=3$) and also $\theta_n \sim n (1-\frac12 \delta)$ as $n\to\infty$ (similarly as in Guess~\ref{guess:ninfinity}).
On the other hand, we also have that $\bbP_x(T_{1,2}>t, T_{1,3} >t) \sim c_{x_1} t^{-(1-\frac12 \delta)}$, so that Guess~\ref{guess:2} does not hold in the case $\delta\leq 0$: we have here that $\bbP_x(T_{1,2}>t) \sim (1+\frac{c_{x_2}}{c_{x_1}}) \bbP_x(T_{1,2}>t, T_{1,3} >t)$ and therefore $\theta_2=\tilde \theta_3 <\theta_3$.

\subsection{Relation to PDEs}
\label{sec:PDEs}

We mention in this section the relation of our question with some PDE problems, which provide other approaches for studying the persistence exponent.
We will not pursue these approaches further since we were not able to obtain any useful information from it.

\paragraph*{Laplace transform of the hitting time.}
In this paragraph we recall the classical fact that the Laplace transform of the  hitting time can be obtained by solving a PDE problem with boundary conditions. In our context, the PDE is not so complicated, but the difficulties lie in the boundary conditions. Let us denote by $\varphi_\lambda (x):=\bbE_x [e^{-\lambda H_n}]$ the Laplace transform of $H_n:= \min\big\{ T_{i,j} ,\ 1\leq i < j \leq n \big\}$, with starting point $X_0=x \in (\mathbb{R}_+)^n$. 
Since the stopping time $H_n$ is the hitting time of the set $\mathcal{A} := \bigcup_{i\neq j} \{x_i=x_j=0\}$ by the $(\bbR_+)^n$-valued diffusion process $(X_t)_{t\geq0}$, we classically have that~$\varphi_\lambda$ solves  
 \begin{equation}
 \label{edp1}
\begin{cases}
 \mathcal{L}^{n}_\delta \varphi (x) =\lambda \varphi (x) \,, & \quad x\in (\bbR_+^*)^n  \,,\\
\varphi(x)=1 \,,  & \quad  x \in \mathcal{A} \,, \\
\lim\limits_{x\to \infty} \varphi(x)=0 \,.
\end{cases}
 \end{equation}
where we recall that $\mathcal{L}^{n}_\delta$ is the generator of $n$ independent Bessels processes, see~\eqref{def:generator}.
 
When $\theta_n <2$, proving that $\bbP_x(H_n>t)\sim c_x t^{-\theta_n}$ as $t \to\infty$ is equivalent to proving that
\[
1-\varphi_\lambda(x) \sim \lambda^\theta h(x) \qquad \text{ as } \lambda\downarrow 0\,,
\]
where $h$ is expected to be $\mathcal{L}$-harmonic. 

Note that, by scale invariance of Bessel processes, we have $\varphi_\lambda(x) =\varphi_1(\lambda x)$, and the goal would thus be to find the behavior of $\varphi_1$  near $0$, where $\varphi_1$ is the ``good'' eigenfunction solving~\eqref{edp1} with $\lambda=1$.

\paragraph*{Link with Quasi-Stationary Distributions.}

There is a link, which is at first hand not so direct, between our problem and questions related to the theory of Quasi-Stationary Distributions (QSD). 
We recall in a nutshell this theory but we refer to \cite{CMSM13} and \cite{CV23} for detailed references.
Let $(Z_t)_{t\geq0}$ be a Markov process on a state space $\mathcal{X}$.
We assume that $\mathcal{X}$  can be decomposed in two parts: $\mathcal{X}_a$, the set of \emph{allowed} states and $\mathcal{X}_f := \mathcal{X} \setminus \mathcal{X}_a$, the set of \emph{forbidden} states, and we let $T:=\inf\{t>0, \ Z_t \in \mathcal{X}_f \}$ be the hitting time of $\mathcal{X}_f$. 
A distribution $\nu$ on $\mathcal{X}_a$ is said to be a Quasi-Stationary Distribution (QSD) if it is invariant under time evolution when the process is conditioned to survive in~$\mathcal{X}_a$, that is such that for all $t>0$ and $A\subset \mathcal{X}_a$,
\[
\bbP_{\nu}( Z_t \in A \mid T>t) = \nu(A).
\]
This condition implies that $T$ is exponentially distributed under $\bbP_\nu$, \emph{i.e.}\ there is some $\theta_\nu>0$ such that $\bbP_\nu (T>t)=e^{-\theta_\nu t}$, and formally the couple $(\nu, \theta_\nu)$ solves the spectral problem 
\[
\mathcal{L}^* \nu= -\theta_\nu \nu,
\] 
where $\mathcal{L}^*$ is the adjoint of the generator $\mathcal{L}$ of $Z_t$ killed when it reaches $\mathcal{X}_f$.  

The basic questions in this theory are the existence of QSD and of the so-called \emph{Yaglom} limits, that is, for some initial distribution~$\mu$, the convergence of the conditional laws $\bbP_\mu ( Z_t \in \cdot \mid T>t)$ towards some QSD measure when $t$ goes to infinity (note that Section~\ref{sec:twoBessels} could be framed in this spirit). 
In general, it is expected that, for all $x\in \mathcal{X}_a$, $\bbP_x ( Z_t \in \cdot \mid T>t)$ converges to $\nu_\star$, where $\nu_\star$ is the \emph{minimal} QSD measure, \emph{i.e.}\ the one associated with the eigenvalue $\theta_\star$ at the bottom of the spectrum of $-\mathcal{L}^*$.
Such a result would give that, for all $x\in \mathcal{X}_a$, 
\[
\bbP_x( T>t)= e^{-\theta_\star t (1+o(1))}\qquad \text{ as } t\to\infty \,.
\]

At first, our problem seems quite different, the hitting time $H_n$ of $\mathcal{A}=\bigcup_{i\neq j} \{x_i=x_j=0\}$ having a heavy-tailed distribution. 
But, as we will see in Section~\ref{Sec:Existence} below, we can perform an exponential time change by considering $\hat{X}_t:= e^{-t} X_{e^t-1}$, which remains a Markov process (it is a $n$-dimensional Cox-Ingersoll-Ross process).
Then, if $\hat{H}_n$ denotes the hitting time of $\mathcal{A}$ by $\hat{X}_t$, we get that having $\bbP(H_n >t)=t^{-\theta_n(1+o(1))}$ as $t\to\infty$ is equivalent to $\bbP(\hat{H}_n>t)=e^{-\theta_n t (1+o(1))}$.
Thus, following the theory of QSD, our persistence exponent is expected to be the bottom of the spectrum of~$-\hat{\mathcal{L}}^*$, the adjoint of the generator of $\hat{X}_t$ killed when it reaches~$\mathcal{A}$.
Unfortunately, up to our knowledge, there is no general result in the QSD theory which can be applied directly to our problem and provide the existence of $\theta_n$ (and the minimal QSD associated).
In our situation, the difficulties come from the fact that we consider a $n$-dimensional diffusion (with $n\geq 2$), taking values in an unbounded set, and also that the \emph{forbidden} set $\mathcal{A}$ is a proper subset of the boundary of the state space $(\bbR_+)^n$.
Note that a QSD theory would provide the existence of a persistence exponent $\theta_n$ and of a Yaglom limit, but \textit{not} the value (or estimates) on the exponent $\theta_n$.
Instead, we prove the existence of $\theta_n$ via some ``elementary'' sub-additive techniques and we estimate $\theta_n$ also via some ``elementary'' techniques.

\paragraph*{Link with a spectral problem on a bounded domain.}

In this paragraph we discuss another approach to obtain $\theta_n$, which exploits the symmetries of the problem and which reduces to a spectral problem for a certain operator on a \textit{bounded} domain. 
The advantage of this approach is that we reduce the number of variables by one, and also that we obtain a diffusion on a bounded domain; the caveat is that the diffusion is harder to study.
We only give an overview of the reduction one could perform and we provide some details in Appendix~\ref{app:spectralpb}

For simplicity, we consider the case $n=3$, and recall that we denote $X_t:=(X_t^{(1)},X_t^{(2)},X_t^{(3)})$. Anticipating a bit with notation, we further define the three elementary symmetric polynomials in the coordinates of $X_t$,
\[
S_t:= X_t^{(1)} + X_t^{(2)} +X_t^{(3)}\,,\quad A_t:=X_t^{(1)} X_t^{(2)} + X_t^{(2)} X_t^{(3)} +  X_t^{(3)} X_t^{(1)}\,,\quad
P_t:= X_t^{(1)} X_t^{(2)} X_t^{(3)}\,,
\] 
which have respective homogeneity $1$, $2$ and $3$. 
Note that, for all $t\geq 0$, $X_t$ is entirely determined, up to some permutation, by $(S_t,A_t,P_t)$. Also, since the Bessels processes are independent we can check that the process $(S_t,A_t,P_t)_{t\geq 0}$ is itself a diffusion process, whose generator can be computed explicitly, see~\eqref{eq:generatorAPT} for a formula.

Expressed with those symmetrical coordinates, the hitting time $H_3$ can be expressed as $H_3:= \inf\{t\geq 0, A_t=0\}$ (notice that $X_t \in \mathcal{A}$ if and only if $A_t=0$).
Moreover, as a consequence of the symmetries of the problem, we can factorize the dynamics of $(S_t,A_t,P_t)$ by $(S_t)_{t\geq 0}$, which plays the role of a ``radial'' process, and some ``angular'' (\emph{i.e.}\ without scaling) process $(\bar{A}_t, \bar{P}_t):=(A_t/S_t^2, P_t/S_t^3)$.
It turns out that one can write the angular ($2$-dimensional) process as a time-changed  diffusion $(U_t,V_t)$, independent of $(S_t)_{t\geq0}$ and whose generator $\bar{\mathcal{L}}$ can also be computed (again, see Appendix~\ref{app:spectralpb} for details) --- this in analogy with what is done in Section~\ref{sec:twoBessels}, see~\eqref{eq:SU}, in the case of $n=2$ Bessels.
Also, one can show that the angular process $(\bar{A}_t, \bar{P}_t)_{t\geq 0}$ evolves in a bounded domain~$\bar{\mathcal{T}} \subset (\bbR_+)^2$ (with boundary) which can be determined explicitly, see Figure~\ref{fig:domain} in Appendix~\ref{app:spectralpb} for an illustration.

Now we can relate the persistence exponent to a spectral problem for the generator $\bar{\mathcal{L}}$ on the \textit{bounded} domain $\bar{\mathcal{T}}$: finding $(\mu, \varphi)$ such that $\bar{\mathcal{L}}\varphi= \mu \varphi$ with $\varphi$ a non-negative function on $\bar{\mathcal{T}}$ that vanishes only when $u=0$, then one should be able to relate the eigenvalue $\mu$ to the persistence exponent $\theta_3$ by the relation $\theta_3(\theta_3-1+3\delta/2)+\mu=0$.
We refer to Appendix~\ref{app:spectralpb} for details, but we were not able to exploit further this approach, the spectral problem seeming out of our reach.

\section{Some preliminaries}
\label{sec:prelim}

\subsection{A comparison theorem}

We state in this section a comparison theorem for Bessel processes with varying dimensions.
The proof is standard and can be found in~\cite[Ch.~6]{MR1011252}.
  We consider here a probability space $(\Omega, \mathcal{F}, \mathbb{P})$ supporting a Brownian motion $(W_t)_{t\geq0}$ and we denote by $\mathbb{F} = (\mathcal{F}_t)_{t\geq0}$ the filtration generated by this Brownian motion, after the usual completions. Let $(D_t^{1})_{t\geq0}$ and $(D_t^{2})_{t\geq0}$ be two $\mathbb{F}$-adapted non-negative processes. Let also $(Z_t^1)_{t\geq0}$ and $(Z_t^2)_{t\geq0}$ be two processes such that, if it exists, \textit{a.s.} for any $t\geq0$,
\[
 Z_t^1 = z_1 + 2\int_0^t \sqrt{Z_s^1} \dd W_s + \int_0^t D_s^1 \dd s \quad \text{and} \quad Z_t^2 = z_2 + 2\int_0^t \sqrt{Z_s^2} \dd W_s + \int_0^t D_s^2 \dd s
\]
for some $z_1, z_2 \geq 0$. We have the following comparison theorem.
\begin{proposition}[Thm.~1.1 in Ch.~6 of~\cite{MR1011252}]
\label{prop:compar_theorem}
 If $z_1 \leq z_2$ and almost surely for any $t \geq0$, $D_t^1 \leq D_t^2$, then almost surely for any $t \geq0$, $Z_t^1 \leq Z_t^2$.
\end{proposition}
%

\subsection{Existence of the persistence exponent}\label{Sec:Existence}

Let us prove~Proposition~\ref{prop-existence-theta} in this section.
First, we perform some  exponential time change of $(X_t)_{t\geq0}$ and consider the process $(\hat{X}_t:= e^{-t} X_{e^{t}-1})_{t\geq 0}$  which is still a Markov process  (known as a Cox-Ingersoll-Ross process, see for instance \cite{GY03}) generated by
\[
\hat{\mathcal{L}}_n^{\delta} := 2 \sum_{i=1}^n x_i \frac{\partial^2}{\partial x_i^2} + (\delta-x_i) \frac{\partial}{\partial x_i} \,.
\]
Then, if we denote $\hat{H}_n:=\inf\{t\geq 0, \exists i\neq j, \hat{X}_t^{(i)}= \hat{X}_t^{(i)}=0 \}$, we naturally have  $\bbP_x( H_n> t)=\bbP_{x}(\hat{H}_n> \log(1+t))$.
Therefore, to prove Proposition~\ref{prop-existence-theta} we simply need to show that, for any $x \in (\bbR_+)^n$,
\begin{equation}\label{limit-theta}
\lim_{t\to\infty} \frac{1}{t}\log \bbP_x(\hat{H}_n> t) = -\theta_n \,.
\end{equation}
Notice also that since $t\mapsto \bbP_x(\hat{H}_n> t)$ is non-increasing, one can consider the limit in~\eqref{limit-theta} only along integers.

Before we prove~\eqref{limit-theta}, let us stress that the limit (if it exists) does not depend on $x$.
For $x,y\in (\bbR_+)^n$, let us write $x\leq y$ if $x_i\leq y_i$ for all $i=1,\ldots, n$. 
Then, by the comparison property of Proposition~\ref{prop:compar_theorem} (applying it componentwise), we obtain that, for any starting point $y\geq x$, $\bbP_{y}( \hat{H}_n> t) \geq \bbP_{x}( \hat{H}_n> t)$.
Therefore, for $x,x'\in (\bbR_+^*)^n$, the Markov property gives that
\begin{align*}
\bbP_{x}(\hat H_n >1+t) & \geq \bbP_{x}\big(\hat H_n >1+t , \hat X_1 \geq x'\big) = \bbE_{x}\Big[\ind_{\{\hat H_n >1 , \hat X_1 \geq x' \}} \bbP_{\hat X_1} \big( \hat H_n >t\big) \Big] \\
& \geq \bbP_x\big( \hat H_n >1 , \hat X_1 \geq x'\big) \bbP_{x'} \big( \hat H_n >t\big)  =: C_{x,x'}\bbP_{x'} \big( \hat H_n >t\big) \,,
\end{align*}
where we have used the comparison inequality for the second line.
This shows that, for any $x,x'\in (\bbR_+^*)^n$ and $t>1$,
\[
C_{x',x} \bbP_{x'}(\hat H_n > t-1)\leq \bbP_{x}(\hat H_n >t) \leq C_{x',x}^{-1} \bbP_{x'}(\hat H_n > t+1 )
\]
so that the limit in~\eqref{limit-theta}, if it exists, does not depend on $x$.

\smallskip
Now, let $x \in (\bbR_+)^n$. To prove~\eqref{limit-theta}, let us introduce, for $t>0$,
\[
q_x(t) :=\bbP_x \big( \hat{H}_n> t, \hat{X}_t \geq x \big) \,.
\]
We now show that $(\log q_x(t))_{t\geq 0}$ is super-additive.
Indeed, by the Markov property, we have that 
\[
 q_x(t+s) = \bbE_x \Big[\ind_{\{\hat{H}_n> t\}} \bbP_{\hat{X}_t}( \hat{H}_n> s, \hat{X}_s \geq x) \Big] \geq \bbE_x \Big[\ind_{\{\hat{H}_n> t, \hat{X}_t\geq x\}} \bbP_{\hat{X}_t}( \hat{H}_n> s, \hat{X}_s \geq x) \Big] \,.
\]
Now, by comparison (applying Proposition~\ref{prop:compar_theorem} componentwise), we obtain that, for any starting point $y\geq x$, $\bbP_{y}( \hat{H}_n> s, \hat{X}_s \geq x) \geq q_x(s)$.
We therefore end up with $q_x(t+s) \geq q_x(t)q_x(s)$ for any $s,t\geq 0$, which shows the super-additivity and thus that the limit
\[
\theta_n := -\lim_{t\to\infty} \frac1t \log q_x(t)
\] 
exists (the limit is taken along integers).
%

\smallskip
We can now compare $q_x(t)$ with the original probability $\bbP_x(\hat H_n >t)$.
First of all, we clearly have that $q_x(t) \leq \bbP_x(\hat H_n > t )$, so that 
$\liminf_{t\to\infty} \frac1t \log \bbP_x(\hat H_n > t) \geq -\theta_n$.

The other bound is a bit more subtle.
Recall that $\mathcal{A} := \bigcup_{i\neq j} \{x_i=x_j=0\}$ and let us define 
\[
\mathcal{A}_{\delta} = \big\{y = (y_1,\ldots, y_n) \in (\bbR_+)^n ,   \exists i<j  \text{ such that } \max(y_i,y_j)< \delta \big\}
\] for some $\delta>0$.
Then, we fix $\gep>0$ and $\delta>0$, and we consider the upper bound
\begin{equation}
\label{eq:splitproba}
\bbP_x(\hat H_n > t ) \leq \bbP_x\big(\hat X_s\in \mathcal{A}_{\delta} \text{ for all } s \in [(1-\gep)t,t-1]  \big) + \bbP_x \big(\hat H_n > \tau_{\delta}, \tau_{\delta} \leq t-1  \big) \,,
\end{equation}
where we have set $\tau_{\delta} := \inf\{ s > (1-\gep)t ,  \hat X_s \notin \mathcal{A}_{\delta}\}$.
We now estimate both probabilities.

For the first one, applying Markov's inequality  iteratively every unit of time, we get that
\begin{multline*}
 \bbP_x\big(\hat X_s\in \mathcal{A}_{\delta} \text{ for all } s \in [(1-\gep)t,t-1\big]  \big)  \leq  \Big( \sup_{y\in \mathcal{A}_{\delta}} \bbP_y\big( \hat X_s\in \mathcal{A}_{\delta} \text{ for all } s \in [0,1] \big) \Big)^{\lfloor \gep t -1\rfloor} \\
  = \exp\big(- C_{\delta} \lfloor \gep t-1 \rfloor \big).
\end{multline*}
where the constant $C_{\delta} = - \log (\sup_{y\in \mathcal{A}_{\delta}} \bbP_y( \hat X_s\in \mathcal{A}_{\delta} \text{ for all } s \in [0,1]))$ goes to $+\infty$ as $\delta \downarrow 0$. 
 Indeed, observe that, for all $y\in \mathcal{A}_{\delta}$,
\begin{align*}
 \bbP_y( \hat X_s\in \mathcal{A}_{\delta} \text{ for all } s \in [0,1])  \leq \bbP_0\Big( \sup_{s\in[0,1]} \min_{i\neq j}\max(\hat X_s^{(i)},\hat X_s^{(j)} ) < \delta^2\Big)\,,
\end{align*}
by comparison.
Now, the upper bound converges to $\bbP_0( \sup_{s\in[0,1]} \min_{i\neq j}\max(\hat X_s^{(i)},\hat X_s^{(j)} ) =0 )$ as $\delta\downarrow 0$, which is equal to $0$.
 \black

For the other probability, let us set 
\[
p_{x, \delta}(s):= \inf_{ y\in \partial \mathcal{A}_{\delta}} \bbP_y(\hat H_n >s, \hat X_s \geq x) \,.
\]
Then, we have that
\[
\bbP_x \big(\hat H_n > \tau_{\delta}, \tau_{\delta} \leq t-1  \big)  \leq \bbE_x \Big[ \ind_{\{\hat H_n > \tau_{\delta},\tau_{\delta} \leq t-1 \}}  \frac{p_{x,\delta}(t-\tau_{\delta})}{ \inf_{s \in [1,\gep t]} p_{x,\delta}(s)} \Big] \,,
\]
and, by the strong Markov property,
\[
\begin{split}
\bbE_x \big[ \ind_{\{\hat H_n > \tau_{\delta},\tau_{\delta} \leq t-1 \}}  p_{x,\delta}(t-\tau_{\delta}) \big] 
& \leq  \bbE_x \big[ \ind_{\{\hat H_n > \tau_{\delta},\tau_{\delta} \leq t-1 \}} \bbP_{\hat X_{\tau_{\delta}}} \big( \hat H_n > t-\tau_{\delta} , \hat X_{t-\tau_{\delta}} \geq x\big) \big] \\
& \leq \bbP_x \big(\hat H_n > t, \tau_{\delta} \leq t-1,  \hat X_t \geq x  \big) \leq q_x(t) \,.
\end{split}
\]

Now, notice that by the Markov property and by comparison (see Proposition~\ref{prop:compar_theorem}), for $s>1$ we have that
$\bbP_y(\hat H_n >s, \hat X_s \geq x) \geq  \bbP_y(\hat H_n >1, \hat X_1 \geq x)  \bbP_x(\hat H_n >s-1, \hat X_{s-1} \geq x)$, so that 
\[
p_{\delta, x}(s) \geq C_{\delta, x} q_x(s-1)\,,
\]
with $C_{\delta, x}:= \inf_{\{ y\in \partial \mathcal{A}_{\delta} \}} \bbP_y(\hat H_n >1, \hat X_1 \geq x) >0$. 

Indeed, for any $y \in \partial \mathcal{A}_{\delta}$, there is at most one~$i$ with $y_i<\delta$: since $\{\hat H_1 >1\} \supset \bigcap_{j\neq i}\{\forall s \in [0,1]\; X_s^{(j)}>0\}$, we get by independence, then by comparison, that
\begin{align*}
\bbP_y\big(\hat H_n >1, \hat X_1 \geq x \big) & \geq 
\bbP_{y_i}\big( \hat X_1^{(i)} \geq x_i \big) \prod_{j\neq i} \bbP_{y_j}\big( \forall s \in [0,1]\; \hat{X}_s^{(j)}>0, \hat X_1^{(j)} \geq x_j \big)\\
& \geq \bbP_{0}\big( \hat X_1^{(1)} \geq  \|x\|_{\infty} \big)  \bbP_{\delta}\big( \forall s \in [0,1]\; \hat{X}_s^{(1)}>0, \hat X_1^{(1)} \geq \|x\|_{\infty} \big)^{n-1} \,,
\end{align*}
which is a positive lower bound on $C_{\delta,x}$.
\black
All together, we obtain that
\[
\bbP_x \big(\hat H_n > \tau_{\delta}, \tau_{\delta} \leq t-1  \big)  \leq C_{\delta,x}^{-1} \;\frac{1}{\inf_{s\in [0,\gep t]} q_x(s)}\; q_x(t)  \,.
\]

Going back to~\eqref{eq:splitproba} and using that $\lim_{t\to\infty} \frac1t \log q_x(t) = -\theta_n$, we conclude that for $t$ sufficiently large (how large may depend on $\gep,\delta,x$), we have $\inf_{s\in [0,\gep t]} q_x(s) \geq e^{-(\theta_n +\gep)\gep t}$ and $q_x(t) \leq e^{-(\theta_n-\gep)t}$ so,
\[
\bbP_x(\hat H_n > t ) \leq e^{ -\frac12 C_{\delta} \gep t} + e^{- ( \theta_n -\gep -\gep \theta_n-\gep^2)t} \,.
\]
Now, for any fixed $\gep>0$, we can choose $\delta$ small enough so that $\frac12 C_{\delta} \gep \geq \theta_n -\gep -\gep \theta_n-\gep^2$, which gives that $\limsup_{t\to\infty} \frac1t \log \bbP_x(\hat H_n > t) \leq - ( \theta_n -\gep -\gep \theta_n-\gep^2)$.
This concludes the proof, since $\gep$ is arbitrary.
\qed

\subsection{General strategy of the proof}
\label{sec:strategy}

We consider a probability space $(\Omega, \mathcal{F}, \mathbb{P})$ supporting $(n+1)$ independent Brownian motions $W^{(0)},W^{(1)}, \ldots, W^{(n)}$. We denote $(\mathcal{F}_t)_{t\geq0}$ the filtration generated by these Brownian motions, after the usual completions. On this filtered probabilty space, we consider $n$ independent squared Bessel processes of dimension $\delta \in(0,1)$, solution of
\[
 \dd X_t^{(i)} = 2 \sqrt{X_t^{(i)}} \dd W_t^{(i)} + \delta \dd t.
\]
Our general strategy is to find some auxiliary one-dimensional stochastic process $(Z_t)_{t\geq 0}$, which hits $0$ exactly at time $H_n$, that we are able to compare with a (time-changed) squared Bessel process, for which the first hitting time of $0$ is well-understood.

More precisely, let $\gp: (\bbR_+)^n \to \bbR_+$ be a smooth function, and define for all $t\geq 0$,
\[
Z_t:= \gp\big( X_t^{(1)},\ldots, X_t^{(n)} \big) =: \gp(X_t)   \,.
\]
Then, using the evolution equation of $X^{(i)}$ and applying It\^o's formula, we obtain the evolution equation of~$Z_t$:
\[
\dd Z_t = 2 \sum_{i=1}^n \sqrt{X_t^{(i)}} \, \frac{\partial \gp}{\partial x_i} (X_t) \, \dd W_t^{(i)} + \Big( \delta \sum_{i=1}^n \frac{\partial \gp}{\partial x_i}(X_t) + 2 \sum_{i=1}^n  X_t^{(i)} \,\frac{\partial^2 \gp}{\partial x_i^2}(X_t) \Big) \dd t \,.
\]
Let us set
\begin{equation}\label{eq:def_y}
 Y_t :=  \sum_{i=1}^n X_t^{(i)} \Big(\frac{\partial \gp}{\partial x_i}(X_t)\Big)^2.
\end{equation}
Recalling that $W^{(0)}$ is a Brownian motion independent from the rest, we define
\[
 \dd W_t :=\bm{1}_{\{Y_t = 0\}}\dd W_t^{(0)} +  \sum_{i=1}^n \bm{1}_{\{Y_t > 0\}}\frac{\sqrt{X_t^{(i)}} \frac{\partial \gp}{\partial x_i} (X_t)}{\sqrt{Y_t}}  \dd W_t^{(i)} \,,
\]
which is an $(\mathcal{F}_t)_{t\geq0}$ Brownian motion since it is a local martingale with quadratic variation $\langle W \rangle_t =t$. Note that whenever $Y_t = 0$ for some $t\geq0$, we have $X_t^{(i)} (\frac{\partial \gp}{\partial x_i}(X_t))^2 = 0$ for every $i \in \{1, \ldots, n\}$. Therefore, if we set 
\[
D^{(1)}_t =\sum_{i=1}^n \frac{\partial \gp}{\partial x_i}(X_t)
\quad \text{and}\quad  
D^{(2)}_t =\sum_{i=1}^n  X_t^{(i)} \,\frac{\partial^2 \gp}{\partial x_i^2}(X_t) \,,
\] 
we get that
\begin{equation}\label{eq:evolZ_1}
 \dd Z_t = 2 \sqrt{Y_t} \dd W_t + \big(\delta D^{(1)}_t + 2D^{(2)}_t\big)\dd t.
\end{equation}
Let us now define the processes $V$ and $D$ by
\begin{equation}
\label{def:VD}
V_t :=  \frac{Y_t}{Z_t} \,,\quad \text{and} \quad D_t := \frac{\delta D^{(1)}_t + 2 D^{(2)}_t}{V_t} \,.
\end{equation}
We will now make the following assumption on the function $\gp$ which will be verified in practice.
\begin{itemize}
 \item[$\mathbf{(H)}$] \namedlabel{hyp:H}{(H)}
 The function $\gp$ is such that \textit{a.s.}\ $\mathrm{Leb}(\{t \geq 0, \: Y_t = 0\}) = 0$ and $V_t < \infty$ for any $t \geq0$.
\end{itemize}
Under this assumption, it turns out that we can rewrite \eqref{eq:evolZ_1} as
\begin{equation}\label{def:evolutionZ}
\dd Z_t = 2 \sqrt{Z_t V_t}\, \dd W_t + D_t V_t \, \dd s \,.
\end{equation}
The advantage of the formulation~\eqref{def:evolutionZ} is that it formally looks like the evolution equation of a \emph{time-changed} square Bessel process, with \emph{varying} dimension $D_t$.
Our objective is now be to find functions $\gp$ (one for the upper bound, one for the lower bound) such that:
\begin{itemize}
\item the function $\gp$ verifies $\gp(x_1,\ldots, x_n) =0$ if and only if\footnote{In fact, one actually need only one of the implication depending on whether one is interested in the upper or the lower bound, but we stick to the \textit{if and only if} formulation for simplicity.} $x_i=x_j=0$ for some $i\neq j$, so that $H_n=\inf\{t >0, Z_t =0\}$;
\item the ``velocity'' $V_t$ and the ``dimension'' $D_t$ can be controlled, namely one can obtain explicit bounds on them.
\end{itemize}
 
\noindent
Let us set
\[
\rho_t := \int_0^{t} V_u \dd u \,,
\]
which corresponds to the time-change in~\eqref{def:evolutionZ}. Thanks to Assumption~\ref{hyp:H} and the definition~\eqref{def:VD} of $V$, we see that \textit{a.s.}\ $\mathrm{Leb}(\{t \geq 0, \: V_t = 0\}) = 0$. 
This implies that $\rho$ is an increasing and continuous time-change. Let us denote by $\tau$ its inverse and let us set $K_t = Z_{\tau_t}$ as well as $\mathcal{B}_t = \int_0^{\tau_t}\sqrt{V_s} \dd W_s$ which is an $(\mathcal{F}_{\tau_t})_{t\geq0}$ Brownian motion. We classically have that
\[
 \dd K_t = 2 \sqrt{K_s} \dd \mathcal{B}_s + D_{\tau_t} \dd t.
\]

\paragraph{For the upper bound on $\bbP_x(H_n>t)$.}

Let us assume here that there is some $\delta_+\in(0,2)$ such that $D_t \leq \delta_+$ uniformly in $t$. Then, if we define
\[
 \dd Q^{(\delta_+)}_t = 2 \sqrt{Q^{(\delta_+)}_t} \dd \mathcal{B}_t + \delta_+ \dd t \,,
\]
which is a squared Bessel process of dimension $\delta_+$, we get by comparison, see Proposition~\ref{prop:compar_theorem}, that \textit{a.s.}\ $K_t \leq Q^{(\delta_+)}_t$ for any $t\geq0$. It follows that $Z_t \leq Q_{\rho_t}^{(\delta_+)}$ for any $t\geq0$. \black Denoting $T_0(Z) := \inf\{t>0, Z_t=0\}$ for a stochastic process $(Z_t)_{t\geq 0}$,  we therefore get that
\[
\bbP_x(H_n>t)  = \bbP_x\big(T_0(Z)>t \big) \leq \bbP_x\big( T_0\big(Q^{(\delta_+)}\big) >\rho_t\big) \,, 
\]
and it remains to control $\rho_t$, and in particular show that it cannot be too small; let us stress that one difficulty is that $(\rho_t)_{t\geq 0}$ is in general not independent from $(Q^{(\delta_+)}_t)_{t\geq 0}$.
We will then need a lemma as follows.

\begin{lemma}
\label{lem:rho-upper}
There is some $\kappa>0$ such that, for any $p\geq 1$ there exists a constant $C_p=C_{x,p}$ such that, for any $t\geq 1$ and any $\gep<1$
\[
\bbP_x\big( \rho_t \leq \gep \, t^{\kappa} \big) \leq  C_p \gep^{p} \,.
\]
\end{lemma}

\noindent
With the help of this lemma, we then get that, for any $p\geq 1$ (large),
\begin{align*}
\bbP_x(H_n>t)\leq \bbP_x\big( T_0\big(Q^{(\delta_+)}\big) >\rho_t\big) 
& \leq \bbP_x\Big( \rho_t \leq t^{-1/\sqrt{p}}\, t^{\kappa}\Big) + \bbP_x\Big( T_0\big(Q^{(\delta_+)}\big) > t^{-1/\sqrt{p}}\, t^{\kappa} \Big) \\
& \leq  C_p t^{- \sqrt{p}} +  c t^{(2-\delta_+)/\sqrt{p}} \;t^{- \kappa(1-\frac12 \delta_+)}\,,
\end{align*}
where we have also used~\eqref{eq:oneprocess} for the last inequality.
This strategy therefore shows that, for any $\eta>0$, choosing $p$ sufficiently large yields $\bbP_x(H_n>t)\leq c' t^{-\theta_- +\eta}$ with $\theta_- := \kappa (1-\frac12 \delta_+)$; here $\delta_+$ is an upper bound on $D_t$ and $\kappa$ gives the scale exponent of $\rho_t$ and appears in Lemma~\ref{lem:rho-upper}.
Since~$\eta$ is arbitrary, this shows that $\theta_n \geq \theta_-$.

\paragraph{For the lower bound on $\bbP_x(H_n>t)$.}

On the other hand, if we assume that there is some~$\delta_-$ such that $D_t \geq \delta_-$ uniformly in $t$, then, just as for the upper bound, we define
\[
 \dd Q^{(\delta_-)}_t = 2 \sqrt{Q^{(\delta_-)}_t} \dd \mathcal{B}_t + \delta_- \dd t
\]
and we get by comparison, thanks to Proposition~\ref{prop:compar_theorem} again, that $K_t \geq Q^{(\delta_-)}_t$ for any $t \geq0$, which yields that $Z_t \geq Q_{\rho_t}^{(\delta_-)}$ for any $t\geq0$. We then get that
\[
\bbP_x(H_n>t)  = \bbP_x\big(T_0(Z)>t \big) \geq \bbP_x\big( T_0\big(Q^{(\delta_-)}\big) >\rho_t\big) \,.
\]
We then now need to show that $\rho_t$ cannot be too large.

\begin{lemma}
\label{lem:rho-lower}
There is some $\kappa>0$ such that, for any $M\geq 1$ there exists a constant $C_p=C_{x,p}$ such that, for any $t\geq 1$ and any $A>1$
\[
\bbP_x\big( \rho_t \geq A \, t^{\kappa} \big) \leq C_p A^{-p} \,.
\]
\end{lemma}

\noindent
Then, with this lemma, we get that
\begin{align*}
\bbP_x(H_n>t)\geq \bbP_x\big( T_0\big(Q^{(\delta_-)}\big) >\rho_t\big) 
& \geq  \bbP_x\Big( T_0\big(Q^{(\delta_+)}\big) > t^{1/\sqrt{p}}\, t^{\kappa} \Big)  - \bbP_x\Big( \rho_t \geq t^{1/\sqrt{p}} \; t^{\kappa}\Big)\\
& \geq   c t^{-(2-\delta_+)/\sqrt{p}}\; t^{- \kappa(1-\frac12 \delta_-)} - C_p t^{-\sqrt{p}}\,,
\end{align*}
where we have again used~\eqref{eq:oneprocess} for the last inequality.
This strategy therefore shows that, for any $\eta>0$, taking $p$ usfficiently large yields that $\bbP_x(H_n>t)\geq c' t^{-\theta_+ -\eta}$ with $\theta_+ := \kappa (1-\frac12 \delta_-)$; here, $\delta_-$ is a lower bound on $D_t$ and $\kappa$ is the one from Lemma~\ref{lem:rho-lower}.
Sicne $\eta>0$ is arbitrary, this shows that $\theta_n \leq \theta_+$.

\begin{remark}
In some cases, one could in theory improve Lemmas~\ref{lem:rho-upper}-\ref{lem:rho-lower} and obtain (stretched) exponential tails for $t^{-\kappa} \rho_t$, see \textit{e.g.}~\eqref{eq:proba-integrales}: this would improve the bounds on  $\bbP_x(H_n>t)$ replacing the $t^{\eta},t^{-\eta}$ by some power of $\log t$. Since we are only interested in the persistence exponent, we do not pursue further this direction.
\end{remark}

\section{Proof of the main result}
\label{sec:proof}

This section consists in applying the strategy of Section~\ref{sec:strategy}, \textit{i.e.}\ choosing the correct functions~$\gp$ for the upper bound and for the lower bound on $\bbP_x(H_n>t)$.

\subsection{Upper bound on $\bbP_x(H_n>t)$}
\label{sec:upper_bound}

For the upper bound, let us first deal with the case $n=3$ for clarity. We turn to the general case $n\geq 3$ afterwards: the strategy is identical but with more tedious calculations.

\subsubsection{The case $n=3$}
Let us consider the functional
\begin{equation}
\label{def:A}
A_t:= X_t^{(1)} X_t^{(2)} + X_t^{(2)} X_t^{(3)} + X_t^{(3)} X_t^{(1)}  =: \gp(X_t) \,,
\end{equation}
and observe that $\gp(x_1,x_2,x_3) = x_1 x_2 +x_2 x_3 +x_3x_1=0$ if and only if $x_1=x_2=0$ or $x_2=x_3=0$ or $x_3=x_1=0$.

Then, let us derive the evolution equation of $(A_t)$, as in~\eqref{def:evolutionZ}.
Denoting $A=x_1 x_2 +x_2 x_3 +x_3x_1$ and $S=x_1+x_2+x_3$, $P=x_1x_2x_3$, straightforward calculations give that 
\[
\sum_{i=1}^3 x_i \Big(\frac{\partial \gp}{\partial x_i} \Big)^2 = AS + 3P \,,
\quad
\sum_{i=1}^3 \frac{\partial \gp}{\partial x_i} = 2 S\,,
\quad
\sum_{i=1}^3 x_i \frac{\partial^2 \gp}{\partial x_i^2} =0 \,.
\]
Hence, recalling the definitions~\eqref{def:VD} and~\eqref{def:evolutionZ}, we obtain
\[
\dd A_t = 2 \sqrt{A_t V_t}\, \dd W_t + D_t V_t \, \dd t \,,
\qquad \text{ with } \ \
V_t  = S_t + 3 \frac{P_t}{A_t} \,,\ \ \text{and}\ \
D_t  = \frac{2 \delta S_t}{S_t + 3 \frac{P_t}{A_t}} \leq 2 \delta \,,
\]
where we denoted 
\[
S_t=X_t^{(1)}+X_t^{(2)}+X_t^{(3)} \quad \text{ and } \quad P_t=X_t^{(1)}X_t^{(2)}X_t^{(3)} \,.
\]

Let us stress that Assumption \ref{hyp:H} is satisfied here and we will show that it is indeed the case in the more general case $n\geq 3$, in the next section below.
Since we have bounded $D_t \leq \delta_- := 2\delta$, in view of Section~\ref{sec:strategy}, it remains to control the time-change. Now, we will show that Lemma~\ref{lem:rho-upper} holds with $\kappa=2$.
Notice that we may simply bound $V_t\geq S_t \geq X_t^{(1)}$, so we only need to show the following: for any $p\geq 1$, there is some $C_p$ such that
\begin{equation}
\label{eq:time-change-upper}
\bbP_x\Big( \int_0^t X_u^{(1)} \dd u \leq  \gep t^{2} \Big) \leq C_p \, \gep^{p} \,.
\end{equation}
We postpone the proof of~\eqref{eq:time-change-upper} to Section~\ref{sec:time-changed}, but it allows us conclude thanks to Section~\ref{sec:strategy} that $\theta_n\geq \theta_- = \kappa (1-\frac12 \delta_+) = 2(1-\delta)$ as announced.

\subsubsection{The general case $n\geq 3$}
\label{sec:upper-general}

Define for $k \in \{1,\ldots, n\}$,
\[
\Pi_k(x) := \sum_{I\subset \{1,\ldots,n\}, |I|=k} x_I \,,\qquad \text{ with }\  x_I := \prod_{i\in I} x_i \,,
\]
and consider the process $Z_t=\Pi_{n-1}(X_t)$, which is indeed equal to $0$ if and only if $X_t^{(i)}=X_t^{(j)}=0$ for some $i\neq j$.
Then, we have that 
\[
\frac{\partial \Pi_{n-1}}{\partial x_i} = \sum_{|I|=n-2,\, i \notin I} x_I = \sum_{j\neq i} x_{\{i,j\}^c}
\]
and in particular $\sum_{i=1}^n \frac{\partial \Pi_{n-1}}{\partial x_i} =  2 \Pi_{n-2}$, where the factor $2$ comes from the fact that each pair $\{i,j\}$ appears twice in the sum. 
Of course we have that $\sum_{i=1}^n x_i \frac{\partial^2 \Pi_{n-1}}{\partial x_i^2} = 0$, so recalling~\eqref{def:VD} we end up with
\[
Y_t = \sum_{i=1}^n X_t^{(i)} \Big(\frac{\partial \Pi_{n-1}}{\partial x_i} \Big)^2(X_t), \quad V_t = \frac{Y_t}{\Pi_{n-1}(X_t)},\quad D_t = \frac{2\delta\, \Pi_{n-2}(X_t)}{V_t} \,.
\]
Let us first show that Assumption \ref{hyp:H} in verified here. First, we observe that for any $x\in(\mathbb{R}_+)^n$,
\begin{equation}\label{eq:y_pi_n}
 \sum_{i=1}^n x_i \Big(\frac{\partial \Pi_{n-1}}{\partial x_i} \Big)^2 = \sum_{i=1}^n \sumtwo{|I| =n-1}{i\in I} \sumtwo{|J|=n-2}{i\notin J} x_I x_J  =\sum_{|I| =n-1} \sum_{|J|=n-2} \vert I \cap J^c \vert  x_I x_J  .
\end{equation}
Moreover, for any $x\in(\mathbb{R}_+)^n$, it is clear that
\begin{equation}\label{eq:pin-1_pin-2}
 \Pi_{n-1}(x) \Pi_{n-2}(x) = \sum_{|I| =n-1} \sum_{|J|=n-2} x_I x_J.
\end{equation}
Since for $I$ and $J$ such that $\vert I\vert =n-1$ and $\vert I\vert =n-2$, $\vert I \cap J^c \vert \leq 2$, we deduce that a.s., for any $t \geq0$, we have $V_t \leq 2 \Pi_{n-2}(X_t) < \infty$. Next, we see from \eqref{eq:y_pi_n} that
\[
 \big\{t \geq0, \: Y_t = 0\big\} \subset \bigcup_{i=1}^n\big\{t \geq0, \: X_t^{(i)} = 0\big\}
\]
which implies that $\mathrm{Leb}(\{t \geq0, \: Y_t = 0\}) = 0$ since for any $i \in \{1, \cdots, n\}$, we classically have that $\mathrm{Leb}(\{t \geq0, \: X_t^{(i)} = 0\}) = 0$.

It remains now to estimate $V_t$. In fact we will show that $V_t \geq \Pi_{n-2}(X_t)$, which gives the bound $D_t \leq 2\delta$, so we again have $\delta_+=2\delta$.
To show this, we observe that for $I$ and $J$ such that $\vert I\vert =n-1$ and $\vert I\vert =n-2$, $\vert I \cap J^c \vert = n-1 +2 - \vert I \cup J^c\vert \geq 1$. Recalling \eqref{eq:y_pi_n} and \eqref{eq:pin-1_pin-2}, we immediately get that $V_t \geq \Pi_{n-2}(X_t)$.

To conclude that $\theta_n\geq (n-1)(1-\delta)$, we show Lemma~\ref{lem:rho-upper} with $\kappa=n-1$.
In fact since $V_t \geq \Pi_{n-2}(X_t) \geq \prod_{i=1}^{n-2} X_t^{(i)}$, we simply need to prove that for all large $p$,
\begin{equation}
\label{eq:time-change-upper2}
\bbP\Big( \int_0^t \prod_{i=1}^{n-2} X_u^{(i)} \dd u \leq  \gep t^{n-1} \Big) \leq C_p\gep^{p}. 
\end{equation}

\begin{remark}
For $1\leq k\leq n$, the first time that $\Pi_{n-k+1}(X_t)$ hits $0$ is the first simultaneous return to zero of $k$ independent Bessel processes, \textit{i.e.}\ the time $(X_t)_{t\geq 0}$ hits the $(n-k)$-dimensional set $\mathcal{A}^{(k)} = \bigcup_{|I|=k} \{x_i=0 \, \forall i\in I\}$, see Remark~\ref{rem:k}. Using the same strategy as above, one can show that the associated persistence exponent $\theta^{(k)}_n$ is larger or equal than $(n-k+1)(1-k \delta/2)$. 
\end{remark}

\subsection{Lower bound on $\bbP_x(H_3>t)$}\label{sec:lower_bound}

As far as the lower bound is concerned, a natural choice of functional would be $Z_t := \gp(X_t)$ with $\gp(x) = (x_1+x_2)(x_2+x_3)(x_3+x_1)$, which is such that $\gp(x)=0$ if and only if $x_i+x_j=0$ for some $i\neq j$.
One can then proceed with the calculations and find that $\kappa=3$ and $\delta_- =2\delta$ so that it gives a bound $\theta_+=3(1-\delta)$.

We are going to give some slightly more optimized functional to improve the upper bound.
Recall the definitions of the functionals $A_t$, $S_t$ and $P_t$ from the previous section.
Then, we define
\[
Z_t := S_t^{a} \Big( A_t - \frac{P_t}{S_t} \Big) =: \gp(X_t)\,.
\]
where $a \in [0,1]$ is some fixed exponent (that depends on $\delta$), to be optimized later on.
Note that for $a=1$, one recovers the functional $\gp(x) = (x_1+x_2)(x_2+x_3)(x_3+x_1)$.
Let us stress however that in the case $a\in (0,1)$, the derivatives of the function $\gp$ are singular at the point $(0,0,0)$ and therefore the It\^o formula is not valid on the time-interval $\mathbb{R}_+$. We will see how we can still apply the strategy outlined in Section \ref{sec:strategy}, but let us first compute the processes $Y$ and $V$ and check that Assumption~\ref{hyp:H} still holds in the present case.

A delicate calculation gives the following (we refer to Appendix~\ref{app:details_calculs} for more details): 
\[
Y :=\sum_{i=1}^3 x_i \Big(\frac{\partial \gp}{\partial x_i} \Big)^2  = S^{2a-1} \Big(A- \frac{P}{S} \Big) \Big( S^2 + a(a+4)A + (1-a)(a+5) \frac{P}{S}   \Big) \,.
\]
It is again clear that we have
$\big\{t \geq0, \: Y_t = 0 \big\} \subset \bigcup_{i=1}^3\big\{t \geq0, \: X_t^{(i)} = 0 \big\}$
so we also have here that a.s.\ $\mathrm{Leb}(\{t \geq0, \: Y_t = 0 \}) = 0$. Recalling now the definition~\eqref{def:VD} of $V$, we find that
\[
 V_t=  S_t^{a-1} Q_t \,, \qquad Q_t:=  S_t^2 + a(a+4)A_t + (1-a)(a+5) \frac{P_t}{S_t} \,.
\]
Notice that $A_t \leq \frac12 S_t^2$ and $P_t \leq \frac16 S_t^3$ so that $Q_t\leq\frac{9}{2} S_t^{2} $ and finally $V_t \leq \frac{9}{2}S_t^{a+1} < \infty$. This shows that Assumption \ref{hyp:H} holds. Let us now compute the process $D$. Some straightforward (but tedious) calculations give the following (again, see Appendix \ref{app:details_calculs} for details):
\[
\begin{split}
D^{(1)} &:= \sum_{i=1}^3 \frac{\partial \gp}{\partial x_i} =  S^{a-1} \Big( 2S^2 + (3a-1) A + 3(1-a) \frac{P}{S} \Big) \,, \\
D^{(2)} & := \sum_{i=1}^3 x_i \frac{\partial^2 \gp}{\partial x_i^2} = S^{a-1} \Big(  a(a+3) A + (1-a)(a+4) \frac{P}{S} \Big)  \,.
\end{split}
\]
Recalling that $D_t = (\delta D_t^{(1)} + +2 D_t^{(2)}) / V_t$, we get that
\[
 D_t= \frac{1}{Q_t} \Big( 2\delta S_t^2 + (\delta (3a-1) + 2a(a+3)) A_t + (1-a)(3\delta + 2(a+4)) \frac{P_t}{S_t} \Big)\,.
\]
Now, we can write $D_t= 2\delta + \frac{1}{Q_t}   \big( f_1(a,\delta)  A_t + f_2(a,\delta) \frac{P_t}{S_t} \big)$, with
\[
f_1(a,\delta)  = 2(1-\delta) a^2 +(6-5\delta)a -\delta\,,\qquad 
f_2(a,\delta) = -2(1-\delta) a^2 - (6-5\delta) a +8-7\delta \,.
\]
We now choose $a=a(\delta)$ such that $f_1(a,\delta)=0$, that is
\[
a(\delta) := \frac{\sqrt{(6-5\delta)^2 +8\delta(1-\delta)} -(6-5\delta)}{4(1-\delta)} \,.
\]
With this choice, we have that $f_2(a,\delta) = 8(1-\delta) \geq 0$, so in particular $D_t \geq 2\delta =: \delta_-$.

Let us now explain how we can apply our strategy even though the function $\gp$ is not $C^2$ on $(\mathbb{R}_+)^3$. If we set $T_0(S) = \inf\{t > 0, \: S_t = 0\}$, we can apply the Itô formula up until this time and the evolution equation \eqref{def:evolutionZ} remains valid on $[0, T_0(S))$: almost surely, for any $t \in [0, T_S)$,
\[
 Z_t = Z_0 + 2\int_0^t\sqrt{Z_s V_s} \dd W_s + \int_0^t D_s V_s \dd s \, .
\]
On the other hand, the time-change $\rho_t = \int_0^t V_s \dd s$ is always well-defined, and, denoting by $\tau$ its inverse, the process $\mathcal{B}_t = \int_0^{\tau_t}\sqrt{V_s} \dd W_s$ is a $(\mathcal{F}_{\tau_t})_{t\geq0}$-Brownian motion. Finally, remembering that $K_t = Z_{\tau_t}$, we get that for any $t \in [0, \rho_{T_0(S)})$,
\[
 K_t = K_0 + 2\int_0^t\sqrt{K_s}\dd \mathcal{B}_s + \int_0^{t}D_{\tau_s} \dd s \, .
\]
Let $Q^{(\delta_-)}$ be the process defined by
\[
 \dd Q^{(\delta_-)}_t = 2 \sqrt{Q^{(\delta_-)}_t} \dd \mathcal{B}_t + \delta_- \dd t \,.
\]
Then, by comparison, we get that a.s. for any $t\in[0,\rho_{T_0(S)})$, $K_t \geq Q_t^{(\delta_-)}$. Since $H_3 \leq T_0(S)$ and since $\rho_{H_3}$ is the first hitting of zero of $K$, it is clear that for any $t \geq0$,
\[
 \{T_0(Q_{(\delta_-)}) > \rho_t\} \subset \{\rho_{H_3} > \rho_t\} = \{H_3 > t\}
\]
and therefore $\mathbb{P}_x(H_3 > t) \geq \mathbb{P}_x(T_0(Q_{(\delta_-)}) > \rho_t)$ as in Section \ref{sec:strategy}.

Then, we need to control the time-change, and we will prove that Lemma~\ref{lem:rho-lower} holds with $\kappa = 2+a$.
For this, remember that $V_t \leq \frac{9}{2} S_t^{a+1}$ and therefore $V_t \leq \frac{9}{2} \, 3^{a+1} \,((X_t^{(1)})^{a+1} +(X_t^{(2)})^{a+1} +(X_t^{(3)})^{a+1})$.
All together, using also a union bound, Lemma~\ref{lem:rho-lower} with $\kappa = 2+a$ follows if we show that for any $b = a+1>0$, for any large $p$
\begin{equation}
\label{eq:time-change-lower}
\bbP_x\Big( \int_0^t (X_u^{(1)})^{b} \dd u \geq  A t^{b+1} \Big) \leq C_p \, A^{-p}  \,.
\end{equation}
Again, we postpone the proof of~\eqref{eq:time-change-lower} to Section~\ref{sec:time-changed}, but we can now conclude thanks to Section~\ref{sec:strategy} that $\theta_n\leq \theta_+ := \kappa (1-\frac12 \delta_+) = 2(1-\delta) +f(\delta)$ with $f(\delta) = a(\delta) (1-\delta)$, as stated in Theorem~\ref{thm:main}.

\begin{remark}
\label{rem:upper-better}
We could try to optimize further the functional $\gp$, for instance considering
$\tilde Z_t := S_t^a ( A_t^{b} - c P_t S_t^{2 b-3})$, for some constants $a,b,c$ to be optimized over (the exponent $2b-3$ ensures that the functional is of homogeneity $a+2b$).
We have used Matematica to help us with the calculations of $V_t,D_t$, and guess a lower bound on $D_t$: it seems that, optimizing over $a,b,c$, one would obtain the following upper bound on the decay exponent:
\[
\theta_+ = 2(1-\delta) + \frac14 \Big(\sqrt{(6-5\delta)^2 +\frac{144\, \delta^2(1-\delta) }{16+9 \delta} } -(6-5\delta) \Big)
\]
However, the calculations are very intricate and the effort seems excessive compared to the improvement of the bound from Theorem~\ref{thm:main} --- we have here $\sup_{\delta\in [0,1]} [ \theta_+ -2(1-\delta)] \approx 0.048$.
\end{remark}

\subsection{Control of the time-change processes: proof of~\eqref{eq:time-change-upper}-\eqref{eq:time-change-upper2} and \eqref{eq:time-change-lower}}
\label{sec:time-changed}

We first show~\eqref{eq:time-change-upper} and \eqref{eq:time-change-lower} before we turn to \eqref{eq:time-change-upper2}, which is an improvement of~\eqref{eq:time-change-upper}.
Such bounds should be classical, but we were not able to find references, so we prove them by elementary (and robust) methods.

\subsubsection{Proof of~\eqref{eq:time-change-upper} and \eqref{eq:time-change-lower}}
\label{sec:chgt-temps}

Let us denote $X_t :=X_t^{(1)}$ for simplicity.
First of all, notice that by scale invariance, we have that, for any $b\geq 0$,
\[
\int_0^t (X_u)^b \dd u \stackrel{(d)}{=} t^{b+1} \int_0^1 (X_u)^b \dd u 
\]
(with a different starting point).
Therefore, taking $b=1$ in~\eqref{eq:time-change-upper} and $b=1+a$ in~\eqref{eq:time-change-lower}, it is enough show that for any $x\in [0,1]$,
\begin{equation*}
\bbP_{x}\Big(\int_0^1 (X_u)^b \dd u  \leq \gep \Big) \leq  C_p \gep^p \,,\qquad 
\bbP_x\Big(\int_0^1 (X_u)^b \dd u  \geq A \Big) \leq C_p A^{-p} \,.
\end{equation*}
In fact, we will show much stronger bounds: we show that there is some $\gamma = \gamma_{b}>0$ and some constant $c>0$ such that
\begin{equation}
\label{eq:proba-integrales}
\bbP_{x}\Big(\int_0^1 (X_u)^b \dd u  \leq \gep \Big) \leq  e^{- c \gep^{-\gamma}} \,,\qquad 
\bbP_x\Big(\int_0^1 (X_u)^b \dd u  \geq A \Big) \leq e^{-c A^{\gamma}} \,.
\end{equation}

Before we prove this, let us show a simple technical lemma that controls the supremum and infimum of a continuous Markov process $(Y_s)_{s\geq 0}$.
The content and the proof of this Lemma are inspired by Etemadi's maximal inequality, see e.g.\ \cite[Thm.~22.5]{Billingsley}.

\begin{lemma}
\label{lem:sup-inf}
Let $(Y_s)_{s\geq 0}$ be a continuous (time-homogeneous) Markov process. Let $A>0$.
Then,  for any $x \in [0,A/2]$,
\[
\bbP_x\big( \sup_{s\in [0,t]} Y_s > A \big)
\leq \bbP_x\big(  Y_t > A/2 \big) + \sup_{s\in [0,t]} \bbP_A\big( Y_s \leq A/2 \big) \,.
\] 
Also, for any $x \geq 4A$,
\[
\bbP_x\big( \inf_{s\in [0,t]} Y_s \leq A \big)
\leq \bbP_x\big(  Y_t \leq x/2 \big) + \sup_{s\in [0,t]} \bbP_A\big( Y_s > x/2 \big) \,.
\] 
\end{lemma}

\begin{proof}
Let us start with the first inequality.
Let $\tau_{A}:= \inf\{s, Y_s=A\}$ and write
\[
\bbP_x\big( \sup_{s\in [0,t]} Y_s > A \big) \leq \bbP_x\big( Y_t > A/2 \big) + \bbP_x\big( \tau_A<t, Y_t \leq A/2\big) \,.
\]
Then, applying the (strong) Markov property at time $\tau_A$, on the event $\{\tau_A<t\}$ we have that $\bbP_x( Y_t \leq A/2 \mid \cF_{\tau_A}) = \gp_t(\tau_A)$ with $\gp_t(s):= \bbP_A(Y_{t-s} \leq A/2)$.
We therefore obtain
\[
\bbP_x\big( \tau_A<t, Y_t \leq A/2\big) = \bbE_x \big[ \ind_{\{\tau_A<t\}} \gp_t(\tau_A) \big] \leq \sup_{s\in [0,t]} \gp_t(s) \,,
\]
which gives the first inequality.

For the second inequality, we use a similar reasoning: we write
\[
\bbP_x\big( \inf_{s\in [0,t]} Y_s < A \big) \leq \bbP_x\big( Y_t < x/2 \big) + \bbP_x\big( \tau_A<t, Y_t \geq  x/2\big) \,.
\]
Then, applying the (strong) Markov property at time $\tau_A$, on the event $\tau_A<t$ we have that $\bbP_x( Y_t \geq x/2 \mid \cF_{\tau_A}) = \tilde \gp_t(\tau_A)$ with $\tilde \gp_t(s):= \bbP_A(Y_{t-s} \geq x/2)$.
We therefore obtain
\[
\bbP_x\big( \tau_A<t, Y_t \geq x/2\big) = \bbE_x \big[ \ind_{\{\tau_A<t\}} \tilde \gp_t(\tau_A) \big] \leq \sup_{s\in [0,t]} \tilde \gp_t(s) \,,
\]
which gives the second inequality.
\end{proof}

Let us now prove the second inequality in~\eqref{eq:proba-integrales}, which is the simpler of the two.
We have
\[
\bbP_x\Big(\int_0^1 (X_u)^b \dd u  \geq A \Big) \leq \bbP_x\Big( \sup_{u\in [0,1]} X_u  \geq A^{1/b} \Big) \leq \bbP_x\Big( X_u  \geq  \tfrac12  A^{1/b} \Big) + \sup_{u\in [0,1]} \bbP_{A^{1/b}}\Big( X_u \geq \tfrac12 A^{1/b} \Big)\,,
\]
where we have used Lemma~\ref{lem:sup-inf} for the last inequality.
Now, recall that $(X_t)_{t\geq 0}$ is a squared Bessel process of dimension $\delta$, so that we have
\begin{equation}
\label{eq:tail-Bessel}
\bbP_x\Big( X_u  \geq \tfrac12 A^{1/b} \Big) \leq  e^{- c A^{1/b}} \,, 
\qquad \bbP_{A^{1/b}}\Big( X_u \leq  \tfrac12 A^{1/b} \Big) \leq e^{- c u^{-1} A^{1/b} } \leq e^{- c A^{1/b} } \,.
\end{equation}
Indeed, these bounds can be easily deduced from the expression of the transition density of squared Bessel processes, see for instance~\cite[Ch. XI, Cor.~1.4]{RY99}. 
This concludes the proof of the second part of~\eqref{eq:proba-integrales}.

\smallskip
We now turn to the first inequality in~\eqref{eq:proba-integrales}.
We let $\gamma = \frac{1}{(4b+4)} <1$, and we write 
\begin{equation}
\label{eq:sup-integrale}
\bbP_x\Big(\int_0^1 (X_u)^b \dd u  \leq \gep \Big) \leq \bbP_x\Big( \sup_{u\in [0,1]} S_u  \leq \gep^{\gamma} \Big)  + \bbP_x\Big( \int_0^1 (X_u)^b \dd u  \leq \gep \,,\, \sup_{u\in [0,1]} X_u  \geq \gep^{\gamma} \Big)\,.
\end{equation}

For the first term in~\eqref{eq:sup-integrale}, we use a rough bound: we use the Markov property at every time $(i-1)\gep^{\gamma}$ for $i\in \{1,\ldots, \gep^{-\gamma}\}$, 
\begin{equation}
\label{eq:supX}
\bbP_x\Big( \sup_{u\in [0,1]} X_u  \leq \gep^{\gamma} \Big) \leq \bigg( \sup_{x\in [0,\gep^{\gamma}]} \bbP_{x}\Big( \sup_{u\in [0,\gep^{\gamma}]} X_u \leq \gep^{\gamma} \Big)  \bigg)^{\gep^{-\gamma}} \leq e^{-c \,\gep^{-\gamma}} \,,
\end{equation}
where for the second inequality we have used that, by scale invariance and comparison, 
\[
\sup_{x\in [0,\gep^{\gamma}]}\bbP_{x}\Big( \sup_{u\in [0,\gep^{\gamma}]} X_u \leq \gep^{\gamma} \Big) = \bbP_{0}\Big( \sup_{u\in [0,1]} X_u \leq 1 \Big) =: e^{-c} \,.\]

Let us now control the second term in~\eqref{eq:sup-integrale}. We set $r:=\frac{2+b}{2+2b}$ and we decompose the probability according to the interval of the form $[(i-1) \gep^{r}, i\gep^{r}]$ where the supremum is attained, on which the infimum also need to be smaller than $\gep^{(1-r)/b}$ (otherwise the integral on this interval would be larger than $\gep$). 
Then, the last term of~\eqref{eq:sup-integrale} is upper bounded by:
\begin{multline*}
\bbP_x\Big( \exists \, i\in \{1,\ldots, \gep^{-r}\} \,, \sup_{u\in [(i-1) \gep^{r}, i\gep^{r}]} X_u  \geq \gep^{\gamma} \,, \inf_{[(i-1) \gep^{r}, i\gep^{r}]} X_u \leq \gep^{(1-r)/b} \Big) \\
 \leq \gep^{-r} \sup_{x\in \bbR_+} \bbP_x\Big( \sup_{u\in [0, \gep^{r}]} X_u  \geq \gep^{\gamma} \,, \inf_{[0, \gep^{r}]} X_u \leq \gep^{2\gamma}  \Big)\,,
\end{multline*}
having used subadditivity and the fact that $(1-r)/b=2\gamma$ for the last inequality (recall that $\gamma:= 1/(4+4b)$).

We then consider two cases for the last probability: either $x \leq  4\gep^{2\gamma}$, or $x\geq 4\gep^{2\gamma}$.
In the case where $x \leq 4\gep^{2 \gamma}$, by comparison and scaling, we bound the probability by
\[
\bbP_{4\gep^{2\gamma}}\Big( \sup_{u\in [0, \gep^{r}]} X_u  \geq \gep^{\gamma}   \Big) = \bbP_4\Big( \sup_{u\in [0,\gep^{r-2\gamma}]} X_u  \geq \gep^{-\gamma}  \Big).
\]
Now from our choices of $\gamma,r$ we have  $r-2\gamma= \frac12 >0$, so $\gep^{r-2\gamma}\leq 1$.  From that we bound the above probability by
\[
\bbP_4\Big( \sup_{u\in [0,1]} X_u  \geq \gep^{-\gamma}  \Big)  \leq \bbP_4 \Big( X_1 > \tfrac12 \gep^{-\gamma}\Big) + \sup_{u\in[0,1]} \bbP_{\gep^{-\gamma}} \Big( X_u < \tfrac12 \gep^{-\gamma}\Big) \,,
\]
using Lemma~\ref{lem:sup-inf}.
Using the fact that $(X_t)_{t\geq 0}$ is a squared Bessel process, we conclude analogously to~\eqref{eq:tail-Bessel} that both probabilities are bounded by $\exp(- c\, \gep^{-\gamma})$.

In the case where $x \geq  2\gep^{2\gamma}$, by comparison and scaling, we bound the probability by
\begin{align*}
\bbP_{4\gep^{2\gamma}}\Big( \inf_{u\in [0, \gep^{r}]} X_u  \leq \gep^{2\gamma}   \Big)  & = \bbP_{4\gep^{-1/2}} \Big( \inf_{u\in [0, 1]} X_u  \leq \gep^{-1/2}   \Big) \\
& \leq \bbP_{4\gep^{-1/2}} \Big( X_1  \leq  2 \gep^{-1/2} \Big) + \sup_{u\in[0,1]} \bbP_{\gep^{-1/2}} \Big( X_u >  2 \gep^{-1/2}\Big) \,,
\end{align*}
where we have also used the fact that  $2\gamma-r= -\frac12$; the second inequality comes from Lemma~\ref{lem:sup-inf}.
Again, analogously to~\eqref{eq:tail-Bessel}, both probabilities are bounded by $\exp(- c\, \gep^{-1/2})$.
This concludes the proof of the second part of~\eqref{eq:proba-integrales}.
\qed

\subsubsection{Proof of~\eqref{eq:time-change-upper2}}

To simplify notation, we let $m=n-2$ and we denote
\[
P_t^{(m)} := \prod_{i=1}^{m} X_t^{(i)} \,.
\]
First of all, note that, by scaling and comparison, we have that
\[
\bbP_x\Big( \int_0^t P_t^{(m)} \dd u \leq  \gep t^{m+1} \Big) \leq \bbP_0 \Big( \int_0^1 P_t^{(m)} \dd u \leq  \gep \Big) \,.
\]
We now show the following lemma, of which point~\textit{(3)} is exactly~\eqref{eq:time-change-upper2}.

\begin{lemma}\label{lem:tech-estimates}
Let $m\geq 1$ be a fixed integer and let $ P_t^{(m)} := \prod_{i=1}^{m} X_t^{(i)}$.
Then, for all $p\geq 1$, we have that there is a constant $C_p=C_{p,m}>0$ such that:
\begin{itemize}
\item[(1)]\quad $\displaystyle\bbE_0\bigg [ \Big ( \sup_{0\leq s <t\leq 1} \frac{\vert P_t^{(m)}-P_s^{(m)} \vert }{(t-s)^{1/4}}\Big)^p\bigg] \leq C_{p} $;\\[5pt]
\item[(2)]\quad $\displaystyle \bbP_0\big( \sup_{s\in[0,1]} P_s^{(m)} \leq \gep \big) \leq C_p\gep^p $;\\[5pt]
\item[(3)]\quad $\displaystyle \bbP_0 \Big ( \int_0^1 P_s^{(m)} \dr s \leq \gep \Big)\leq C_p \, \gep^p$.
\end{itemize}
\end{lemma}

\begin{proof}
First of all, let us simplify notation and write $P_t := P_t^{(m)}$ and $\bbP:=\bbP_0$.
To prove the first point, it suffices to show (by the Kolmogorov criterion, see \cite[Thm.~2.1]{RY99}) that for all $p\geq 1$,
\begin{equation}
\label{eq:Kolmogorov-criterion}
\forall 0\leq s<t\leq 1,  \quad \bbE \big[ \vert P_t-P_s\vert^p \big]\leq c_p \vert t-s \vert^{p/2} \,.
\end{equation}
From the Itô formula, we have that
\begin{equation}\label{eq:ito_produit}
 P_t-P_s= 2\int_s^t \sum_{i=1}^{m} \sqrt{X_u^{(i)}}\prod_{j\neq i} X_u^{(j)} \dr W^{(i)}_u + \delta \int_s^t \sum_{i=1}^{m}\prod_{j\neq i} X_u^{(j)} \dr u \,, 
\end{equation}
and using Burkholder-Davis-Gundy's inequality (see e.g.\ \cite[Thm.~4.1]{RY99}), we get that 
\[
\bbE \big[ \vert P_t-P_s\vert^p \big]\leq c_p' \bbE \Big [ \Big ( \int_s^t \sum_{i=1}^{m} X_u^{(i)}\prod_{j\neq i} \big (X_u^{(j)} \big)^2\dr u\Big)^{p/2} \Big ]+ c_p'' \bbE \Big[ \Big (\int_s^t \sum_{i=1}^{m}\prod_{j\neq i} X_u^{(j)} \dr u\Big)^p \Big] \,.
\]
We finally obtain~\eqref{eq:Kolmogorov-criterion} by dominating, into the integrals, all the $X^{(i)}$ by their supremum on $[0,1]$ (which are independent and admit a moment of order $p$ for all $p\geq1$).

We now prove points \textit{(2)}-\textit{(3)} by iteration on $m\geq 1$. The case $m=1$ has already been treated in Section~\ref{sec:chgt-temps}, so we now take $m\geq 2$ and suppose that points \textit{(2)}-\textit{(3)} hold for $m-1$.

Let us start to show that point~\textit{(2)} holds for $m$.
First of all, observe from \eqref{eq:ito_produit} that $P_t$ is a time-changed square Bessel process $Q^{(\delta)}$ of dimension $\delta$, \textit{i.e.}\ $P_{\tau_t}= Q^{(\delta)}_{t}$, where $\tau_t$ is the inverse of
\[
\rho_t :=\int_0^t \sum_{i=1}^{m-1} \prod_{j\neq i} X_u^{(j)} \dr u \,.
\] 
Then we can write
\begin{align*}
\bbP \Big(  \sup_{[0,1]}P_s \leq \gep \Big) =\bbP\Big( \sup_{[0,\rho_1]} Q^{(\delta)}_s \leq \gep \Big) \leq \bbP\big(\rho_1 \leq \sqrt{\gep} \big) + \bbP\Big( \sup_{[0, \sqrt{\gep}]} Q^{(\delta)}_s \leq \gep \Big) \,.
\end{align*}
By scaling, the second term equals $\bbP(\sup_{[0, 1]} Q^{(\delta)}_s \leq \sqrt{\gep})$ which is bounded by $e^{-c \sqrt{\gep}}$ as seen in~\eqref{eq:supX}. 
Moreover, since $\rho_1$ is the integral of products of $m-1$ independent Bessel processes, one can use point \textit{(3)} with $m-1$ to get that, for any $p\geq 1$, we have $\bbP(\rho_1 \leq \sqrt{\gep} ) \leq C_p' \gep^p$.
This proves point \textit{(2)} for $m$.

We now turn to point \textit{(3)}.
Let us denote
\[
P^*:=\sup_{s\in[0,1]} P_s, \quad \mathrm{and} \quad K:=\sup_{0\leq s<t\leq 1}\frac{\vert P_t-P_s\vert}{\vert t-s\vert^{1/4}} \,,
\]
and let $t^*\in [0,1]$ be such that $P^*=P_{t^*}$. Since by definition of $K$ we have $ P_t\geq P^*-K \vert t-t^*\vert^{1/4}$, we obtain that for all $\gep>0$,
\[
\int_0^1 P_s \dr s \geq \int_{0}^{1} \ind_{\{\vert s-t^*\vert \leq \frac12\gep^{7/8} \}} P_s \dr s \geq \gep^{7/8}P^*-K \gep^{35/32} \geq \gep^{7/8}P^*-K \gep^{9/8} \,.
\]
We therefore get that
\begin{align*}
\bbP \Big(\int_0^1 P_s \dr s\leq \gep \Big) &\leq \bbP\big( P^* \leq 2 \gep^{1/8} \big)+\bbP\big( P^* \geq 2 \gep^{1/8}, \gep^{7/8}P^*-K \gep^{9/8}\leq \gep \big) \\
&\leq \bbP( P^* \leq 2 \gep^{1/8})+\bbP ( K\geq \gep^{-1/8}) \,.
\end{align*}
Hence, using point \textit{(2)} for the first probability and Markov's inequality and point \textit{(1)} for the second one, we obtain that for any $p\geq 1$, both terms are bounded by $C_p\gep^p$.
This proves point~\textit{(3)} for $m$, concludes the recursion and proves the Lemma.
\end{proof}

\begin{appendix}

\section{Calculations from Section \ref{sec:lower_bound}}\label{app:details_calculs}

In this section, we give some details on the tedious calculations from Section~\ref{sec:lower_bound}. We recall that the function $\varphi$ is defined as
$\varphi(x) = S^a\big(A - \frac{P}{S}\big)$
where $a \in[0,1]$ and
\[
 S = x_1 + x_2 + x_3, \qquad A = x_1x_2 + x_2x_3 + x_3x_1, \qquad P = x_1x_2x_3.
\]
Our aim here is to show the three following identities
\begin{align}
 \sum_{i=1}^3 & \frac{\partial \gp}{\partial x_i}  =  S^{a-1} \Big( 2S^2 + (3a-1) A + 3(1-a) \frac{P}{S} \Big) \,, \label{eq:first_ident} \\
 \sum_{i=1}^3 & x_i \frac{\partial^2 \gp}{\partial x_i^2}  = S^{a-1} \Big(  a(a+3) A + (1-a)(a+4) \frac{P}{S} \Big)  \,, \label{eq:sec_ident} \\
 \sum_{i=1}^3 & x_i \Big(\frac{\partial \gp}{\partial x_i} \Big)^2  = S^{2a-1} \Big(A- \frac{P}{S} \Big) \Big( S^2 + a(a+4)A + (1-a)(a+5) \frac{P}{S}   \Big) \label{eq:third_ident} \,.
\end{align}
Since $\varphi$ is a function of $(S, A, P)$, we rely on the chain-rule formula to compute the derivatives of $\varphi$. More precisely, we have
\[
 \frac{\partial \gp}{\partial x_i} = \frac{\partial \gp}{\partial S}\frac{\partial S}{\partial x_i} +\frac{\partial \gp}{\partial A}\frac{\partial A}{\partial x_i} + \frac{\partial \gp}{\partial P}\frac{\partial P}{\partial x_i} =: H_i + F_i + G_i \,,
\]
where
\[
 H := H_i = aS^{a-1}A + (1 - a)S^{a-2}P, \quad F_i = S^a(x_{i+1} + x_{i+2}), \quad G_i = -S^{a-1}x_{i+1}x_{i+2}.
\]
Here, we used the convention that $x_4 = x_1$ and $x_5 = x_2$. Regarding \eqref{eq:first_ident}, we get
\[
 \sum_{i=1}^3 \frac{\partial \gp}{\partial x_i}  = 3 H + 2S^{a +1} - S^{a-1}A = S^{a-1} \Big( 2S^2 + (3a-1) A + 3(1-a) \frac{P}{S} \Big)
\]
Let us now compute \eqref{eq:third_ident}. We start by writing
\[
 \sum_{i=1}^3 x_i \Big(\frac{\partial \gp}{\partial x_i} \Big)^2 = \sum_{i=1}^3 x_i \Big(H^2 + F_i^2 + G_i^2 + 2HF_i + 2HG_i + 2F_i G_i \Big).
\]
Then, computing carefully all of the above six terms, we see that
\[
 \sum_{i=1}^3 x_i H^2 = S^{2a-1}\Big(a^2A^2 + 2a(1-a)A\frac{P}{S} + (1-a)^2 \frac{P^2}{S^2}\Big), \quad \sum_{i=1}^3 x_iF_i^2 = S^{2a-1} S^2\Big(A + 3\frac{P}{S}\Big),
\]
and
\[
 \sum_{i=1}^3 x_iG_i^2 = S^{2a-1}A\frac{P}{S}, \qquad \sum_{i=1}^3 2x_iHF_i = 4S^{2a - 1} \Big(aA^2 + (1-a)A \frac{P}{S}\Big)
\]
and
\[
 \sum_{i=1}^3 2x_iHG_i = -6 S^{2a-1}\Big(aA\frac{P}{S} + (1-a)\frac{P^2}{S^2}\Big), \qquad \sum_{i=1}^3 2x_iF_iG_i = -4S^{2a-1}S^2 \frac{P}{S}.
\]
Recombining all the terms, one can easily check that \eqref{eq:third_ident} holds. Let us now compute the second derivatives of $\gp$. Differentiating in chain with respect to $(S, A, P)$, we get
\[
 \frac{\partial H}{\partial x_i} = a(a-1)S^{a-2}A +aS^{a-1}(x_{i+1} + x_{i+2}) + (1-a)(a-2)S^{a-3}P  + (1-a)S^{a-2}x_{i+1}x_{i+2}
\]
and
\[
 \frac{\partial F_i}{\partial x_i} = aS^{a-1}(x_{i+1} + x_{i+2}), \qquad \frac{\partial G_i}{\partial x_i} = (1-a)S^{a-2}x_{i+1}x_{i+2}.
\]
With these identities at hand and since $\frac{\partial^2 \gp}{\partial x_i^2} = \frac{\partial H}{\partial x_i} + \frac{\partial F_i}{\partial x_i} + \frac{\partial G_i}{\partial x_i}$, one gets that
\[
\sum_{i=1}^3 x_i  \frac{\partial^2 \gp}{\partial x_i^2}= S^{a-1} \Big( a(a-1)A +2a A +(1-a)(a-2)\frac{P}{S}+3(1-a)\frac{P}{S}  + 2a A + 3 (1-a)  \frac{P}{S}\Big) \,,
\]
which concludes that~\eqref{eq:sec_ident} holds.

\section{Relation with a spectral problem on a bounded domain}
\label{app:spectralpb}

We now come back to the last discussion in Section~\ref{sec:PDEs} about the relation with a spectral problem for a certain operator on a bounded domain.
Recall the definitions of the symmetric polynomials
\[
S_t:= X_t^{(1)} + X_t^{(2)} +X_t^{(3)}\,,\quad A_t:=X_t^{(1)} X_t^{(2)} + X_t^{(2)} X_t^{(3)} +  X_t^{(3)} X_t^{(1)}\,,\quad
P_t:= X_t^{(1)} X_t^{(2)} X_t^{(3)}\,,
\] 
and that $H_3 = \inf\{t\geq 0, A_t=0\}$.
The generator $\tilde{\mathcal{L}}$ of $(S_t,A_t,P_t)_{t\geq 0}$ can be computed explicitly and (after calculation) it can be expressed as follows: for a $C^2((\bbR_+^*)^3)$ function $\psi(s,a,p)$ with compact support,
\begin{equation}
\label{eq:generatorAPT}
\begin{split}
\tilde{\mathcal{L}}\psi = 2s\partial_{ss}^2 \psi +2(sa+3p)\partial_{aa}^2\psi+ 2ap\partial_{pp}^2\psi+8a\partial_{sa}^2\psi +& 12p\partial_{sp}^2\psi + 8 ps\partial_{ap}^2\psi\\
&+ 3\delta\partial_s \psi+ 2\delta s\partial_a \psi+ \delta a \partial_p \psi \,.
\end{split}
\end{equation} 

We can now \textit{factorize} the dynamics between a ``radial'' process $(S_t)_{t\geq 0}$ and an ``angular'' process $(\bar{A}_t, \bar{P}_t):=(A_t/S_t^2, P_t/S_t^3)$, similarly to what is done in Section~\ref{sec:twoBessels}.
Indeed, if in~\eqref{eq:generatorAPT} we take a function $\psi$ of the form $\psi(s,a,p)=\phi(s)\varphi(a/s^2,p/s^3)$, we obtain after calculations  that
\begin{equation}
\label{split}
\tilde{\mathcal{L}}\psi(s,a,p)= \varphi(a/s^2,p/s^3)\mathcal{L}^{1}_{3\delta}\phi(s)+\frac{\phi(s)}{s}\bar{\mathcal{L}}\varphi(a/s^2,p/s^3),
\end{equation}
where $\mathcal{L}^{1}_{3\delta}$ is the generator of a Bessel process of dimension $3\delta$ in $\bbR_+$ and $\bar{\mathcal{L}}$ is given by
\begin{multline*}
\bar{\mathcal{L}}\varphi:=  2(u(1-4u)+3v)\partial_{uu}^2 \varphi+2v(u-9v)\partial_{vv}^2+8v(1-3u)\partial^2_{uv} \\
+2\big ( \delta(1-3u)-2u \big)\partial_u \varphi+\big(\delta(u-9v)-12v\big)\partial_v \varphi \,.
\end{multline*}
We stress that this decomposition makes sense for $S_0\neq 0$ and for times $t<\inf\{u\geq 0, S_u=0\}$; in particular it makes sense before the hitting time $H_3$ (since $S_t=0$ implies that $A_t=0$).

Splitting $\tilde{\mathcal{L}}$ as in \eqref{split} means that the angular process $(\bar{A}_t, \bar{P}_t)$ is a time-changed (by $\int_0^t S_u^{-1} \dr u$)  diffusion $(U_t,V_t)$ generated by $\bar{\mathcal{L}}$ and independent of $(S_t)_{t\geq0}$.
Now, one can check that the angular process $(\bar{A}_t, \bar{P}_t)_{t\geq 0}$ evolves in a \textit{bounded} domain of $(\bbR_+)^2$ (with boundary), which is determined by computing the determinant of the principal symbol of $\bar{\mathcal{L}}$. After calculations, one can verify that the angular process lives in a ``curved'' triangle $\bar{\mathcal{T}}$ described by 
\[
\bar{\mathcal{T}}:=\big\{ (a,p)\in [0,+\infty[^2:  p(-4  a^3 + a^2 + 18  a  p - 4 p - 27 p^2) \geq 0 \big\} \,.
\]
We provide in Figure~\ref{fig:domain} an illustration of the domain $\bar{\mathcal{T}}$.

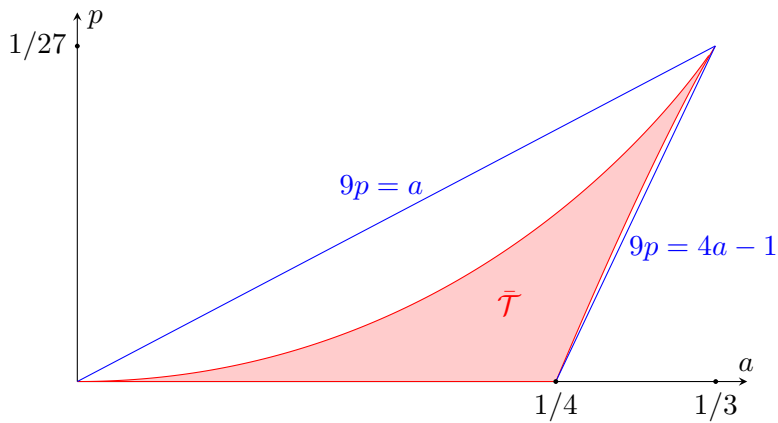
\begin{figure}
\begin{center}
\begin{tikzpicture}[thick,scale=0.8]
   \begin{axis}[
            xmin=0, xmax=1.05, 
            ymin=0, ymax=1.1,
            axis x line = center, 
            axis y line = center,
            xtick = \empty,
            ytick = \empty,
            width=12.6cm,
			height=7.7cm,
            ] 
       \addplot[red, name path=A, samples=100, domain=0.75:1] {3*x-2-2*sqrt(1-3*x+3*pow(x,2)-pow(x,3)))};
       \addplot[red, name path=B,  samples=100, domain=0:1] {3*x-2+2*sqrt(1-3*x+3*pow(x,2)-pow(x,3)))};
      \addplot [red!20] fill between [of=A and B, soft clip={domain=0:1}];
       \addplot[red, samples=200, domain=0:0.75] {0};
       \addplot[blue]coordinates {(0.75,0)(1, 1)};
       \addplot[blue]coordinates {(0,0)(1,1)};
    \end{axis}
    \fill (7.87,0) circle (0.4mm) node[below] {$1/4$};
    \fill (10.5,0) circle (0.4mm) node[below] {$1/3$};
    \fill (0,5.56) circle (0.4mm) node[left] {$1/27$};
    \draw[red] (7.1,1.3) node {$\bar{\mathcal{T}}$};
    \draw[blue] (5,3.21) node {$9p=a$};
    \draw[blue] (10.3,2.2) node {$9p=4a-1$};
    \draw (11,0) node [above]{$a$};
    \draw (0,6) node [right]{$p$};
\end{tikzpicture}
\caption{\small Illustration of the domain $\bar{\mathcal{T}}$ in which the ``angular'' process $(\bar{A}_t, \bar{P}_t))_{t\geq 0}$ lives.
The hitting time $H_3$ is equal to $\inf\{t\geq 0\,, \bar A_t =0 \}$, so it is related to the hitting time of the vertex $(0,0)$ in $\bar{\mathcal{T}}$.}
\label{fig:domain}
\end{center}
\end{figure}

Now we can relate our question about the persistence exponent $\theta_3$ to a spectral problem for the generator $\bar{\mathcal{L}}$ on the bounded domain $\bar{\mathcal{T}}$, in the following way.
The idea is to find a non-negative function $\psi$  which is null only when $a=0$ and which is $\tilde{\mathcal{L}}$-harmonic, \textit{i.e.}\ such that $\tilde{\mathcal{L}}\psi=0$. 
With this function~$\psi$ at hand, one obtains that $\psi(S_t,A_t,P_t)$ is a time-changed Brownian motion\footnote{Note that our strategy, outlined in Section~\ref{sec:strategy}, was to compare the functional $Z_t=\varphi(X_t)$ to a time-changed Bessel process (rather than a Brownian motion here).}, which hits $0$ only when $X_t$ hits $\mathcal{A}$.
In view of the factorization property described above, we can look for $\psi$ in the form $\psi(s,a,p)= s^{\theta}\varphi(a/s^2,p/s^3)$ (with $\theta$ and $\varphi$ to be determined): by the splitting~\eqref{split} given above, $\psi$ is $\tilde{\mathcal{L}}$-harmonic if and only if $\varphi$ verifies the eigenvalue problem $\bar{\mathcal{L}}\varphi= \mu \varphi$, where $\mu$ is such that $\theta(\theta-1+3\delta/2)+\mu=0$.
All together, one ends up with the following time-changed Brownian motion:
\[
\psi(S_t,A_t,P_t) = S_t^{\theta}\varphi(\bar{A}_t, \bar{P}_t )= \mathcal{B}_{\rho_t}\,,
\]
where $\rho_t$ is given by 
$t\mapsto\rho_t = \int_0^t 4 S_r^{2\theta-1}g(\bar{A}_r, \bar{P}_r) \dr r$, and with
\begin{multline*}
g(u,v):=\theta^2 \varphi(u,v)^2 + (u(1-4u)+3v) (\partial_u \varphi (u,v) )^2 \\
+ (u-9v)v (\partial_v \varphi (u,v) )^2 +4 v(1-3u)\partial_u \varphi (u,v) \partial_v \varphi (u,v) \,.
\end{multline*}
Then, we have $\bbP(H_3>t)=\bbP(T^\mathcal{B} > \rho_t)$ where $T^\mathcal{B}$ is the hitting time of $0$ by a Brownian motion~$\mathcal{B}$. Since $\rho_t$ scales like $t^{2\theta}$, one should get (after controlling the time-change), that $\bbP(H_3>t)$ behaves like $t^{-\theta}$ as $t\to \infty$.

To summarize, if one finds $(\mu, \varphi)$ such that $\bar{\mathcal{L}}\varphi= \mu \varphi$ with a function $\varphi : \bar{\mathcal{T}} \to \bbR_+$ which vanishes only when the first coordinate is zero, then applying our strategy one should obtain that the persistent exponent~$\theta_3$ is the solution of $\theta_3(\theta_3-1+3\delta/2)+\mu=0$.

\end{appendix}

\paragraph*{Acknowledgements.}
The authors are grateful to Frank den Hollander for suggesting (and discussing) this apparently simple but challenging question.
We also would like to thank Nicolas Fournier for many enlightening discussions. 
Q.B.\ acknowledges the support of Institut Universitaire de France and ANR Grant Local (ANR-22-CE40-0012).
L.B.\ is funded by the ANR Grant NEMATIC (\href{https://anr.fr/Projet-ANR-21-CE45-0010}{ANR-21-CE45-0010}).

\bibliographystyle{abbrv}
\bibliography{biblio.bib}

\end{document}